\documentclass[11pt]{article}

\usepackage{math-nate}
\usepackage{hyperref}
\usepackage{url}
\usepackage{fullpage}
\theoremstyle{plain}
\newtheorem{assumption}{Assumption}
\numberwithin{equation}{section}

\usepackage{appendix}
\usepackage[numbers]{natbib}
\usepackage{comment}
\usepackage{microtype}

\newcommand{\esssup}{\operatorname{esssup}}
\newcommand{\essinf}{\operatorname{essinf}}

\renewcommand{\Cap}{\operatorname{Cap}}
\newcommand{\loc}{\mathrm{loc}}
\newcommand{\grad}{\nabla}

\newcommand{\inner}[2]{\left( #1 , #2 \right)}

\begin{document}


\title{Widder's representation theorem for symmetric local Dirichlet
spaces} 
\author{Nathaniel Eldredge and Laurent Saloff-Coste}
\maketitle

\begin{abstract}
In classical PDE theory, Widder's theorem gives a representation for
nonnegative solutions of the heat equation on $\R^n$.  We show that an
analogous theorem holds for local weak solutions of the canonical
``heat equation'' on a symmetric local Dirichlet space satisfying a
local parabolic Harnack inequality.
\end{abstract}

\tableofcontents

\section{Introduction}

The purpose of this article is to give an extension of Widder's
theorem \cite{widder44}, which gives a representation for nonnegative
solutions of the heat equation, to the general setting of symmetric,
strictly local, regular Dirichlet spaces.

To motivate the theorem, consider the Cauchy problem for the classical
heat equation on $\R^n$:
\begin{equation}\label{classical-cauchy}
  \begin{split}
    \partial_t u(t,x) - \frac{1}{2} \Delta u(t,x) &=
    0, \quad x \in \R^n,\, t > 0 \\
    u(0,x) &= f(x), \quad x \in \R^n.
  \end{split}
\end{equation}
When (\ref{classical-cauchy}) is introduced in an elementary PDE course, an
immediate question is whether solutions of (\ref{classical-cauchy})
are unique.  The answer, of course, is no: explicit examples of
nonzero functions $u$ satisfying (\ref{classical-cauchy}) with $f=0$
are known, and can be found in \cite{rosenbloom-widder} and references
therein.  However, such counterexamples generally have some sort of
bad behavior.  For instance, they grow rapidly as $\abs{x} \to
\infty$; if one requires certain growth conditions, such as the
requirement that 
\begin{equation} \label{growth}
\abs{u(t,x)} \le C e^{c \abs{x}^2},
\end{equation}
then it is well known that there is a unique solution of
(\ref{classical-cauchy}) which satisfies (\ref{growth}).  See,
e.g. \cite[Theorem 2.3.6]{evans-pde-book}.

Another sort of bad behavior that these counterexamples exhibit is
that they are unbounded below.  If we think of the heat equation as a
model for heat flow, then such solutions are non-physical, since
temperatures cannot be less than absolute zero.  Thus, it is natural
to restrict our attention to \emph{nonnegative} solutions of
(\ref{classical-cauchy}), and ask whether this is sufficient to
ensure uniqueness.  In 1944, D.~Widder showed in \cite{widder44} that
the answer is affirmative.  This is a satisfying result, since the
hypothesis of nonnegativity seems more natural and less arbitrary than
growth conditions such as (\ref{growth}).

Widder also showed that every nonnegative solution $u$ of the
classical heat equation in $\R^n$ for times $t > 0$ is of the form
\begin{equation}\label{classical-rep}
  u(t,x) = P_t \nu(x) := \int_{\R^n} p(t,x,y) \nu(dy)
\end{equation}
for some unique positive Radon measure $\nu$, where $p(t,x,y) =
\frac{1}{(2 \pi t)^{n/2}} e^{-|x-y|^2/2t}$ is the classical heat
kernel.  One can interpret this result as saying that any nonnegative
solution of the heat equation for times $t > 0$ must have evolved from
some initial temperature distribution $\nu$ (which may be singular).
This result was later extended to nonnegative classical solutions
\cite{krzyzanski64} and weak solutions \cite{aronson68,aronson68add}
of more general second-order parabolic equations on $\R^n$.

To extend Widder's result to more general spaces than $\R^n$, one must
first notice that uniqueness of nonnegative solutions of the Cauchy
problem may fail.  For example, if $\Omega$ is a bounded open subset
of $\R^n$, then there are many nonnegative classical solutions of
\begin{equation}\label{classical-cauchy-omega}
  \begin{split}
    \partial_t u(t,x) - \frac{1}{2} \Delta u(t,x) &=
    0, \quad x \in \Omega,\, t > 0 \\
    u(0,x) &= f(x), \quad x \in \R^n.
  \end{split}
\end{equation}
For instance, we might take
\begin{equation*}
  u(t,x) = 
  \begin{cases}
    p(t-t_0, x, x_0), & t > t_0 \ge 0 \\
    0, & t \le t_0
  \end{cases}
\end{equation*}
for some $x_0 \notin \Omega$.  The problem is that
(\ref{classical-cauchy-omega}) does nothing to exclude the possibility
of heat entering through the boundary of $\Omega$.  Thus we certainly
cannot hope to represent a solution $u$ as $u(t,x) = P_t \nu(x)$.  A
more appropriate representation for $u$ might be
\begin{equation}\label{rep}
  u(t,x) = P_t \nu(x) + h(t,x)
\end{equation}
where $h$ is another nonnegative solution of the heat equation which
vanishes at time $0$.  This is the type of result we shall prove in
the present paper.  We shall return to the question of uniqueness in
Section \ref{unique-sec}.

In 1992, Ancona and Taylor \cite{ancona-taylor} proved a
representation result of the form (\ref{rep}) in the language of
abstract potential theory, where solutions of a parabolic equation on
a manifold are considered as a sheaf of functions satisfying certain
properties.  A key ingredient in their proof is the assumption that
solutions satisfy a parabolic Harnack inequality.  It is worth noting
that their results are able to include solutions to equations of the
form $\partial_t u(t,x) - L u(t,x) = 0$ where the
second-order operator $L$ need not be elliptic but can be
hypoelliptic.

In recent years, attention has focused on the notion of Dirichlet
spaces (see section \ref{basic-def-sec}) as a setting for the study of
potential theory.  In this setting, one takes as the underlying space
a metric measure space or similar object; in particular, no
differentiable structure is assumed.  However, the space carries
enough structure to allow one to define a notion of a solution to a
canonical ``Laplace equation'' or ``heat equation,'' and in particular to
study functional inequalities for such solutions, such as Harnack and
Poincar\'e inequalities.  Since a Dirichlet space also carries a
canonical stochastic process, one is also able to exploit tools from
probability and probabilistic potential theory.  

The main result of this paper is to prove that so-called local weak
solutions of the heat equation on a Dirichlet space, under certain
assumptions, are given by a representation of the form (\ref{rep}).
The proof is along similar lines to that of \cite{ancona-taylor}, but
the details are quite different.  Along the way, we shall obtain
several useful properties of local weak solutions.

For related results in a variety of other settings, see
\cite{mair-taylor,
saloff-coste-uniformly-elliptic, 
mendez-murata,
murata-2007}.

\section{Definition of local weak solutions}\label{lws-def-sec}\label{basic-def-sec}

Let $(X,d,\mu)$ be a \textbf{metric measure space}: $(X,d)$ is a
metric space, and $\mu$ is a positive Radon measure on $X$.  We
further assume that $X$ is separable and locally compact; it follows
that $X$ is Polish and that every finite Borel measure is
automatically Radon.

Let $(\mathcal{E},\mathbb{D})$ be a \textbf{Dirichlet form} on
$L^2(X,\mu)$: a closed, Markovian, positive, bilinear form.  (We refer
the reader to \cite{fukushima-et-al} for the definition of a Dirichlet
form and of the following properties.)  The quintuple
$(X,d,\mu,\mathcal{E},\mathbb{D})$ will be called a \textbf{Dirichlet
space}.  We assume that $(\mathcal{E},\mathbb{D})$ is
\textbf{symmetric}, \textbf{regular}, and \textbf{strictly local}.  We
shall write $\mathcal{E}(f)$ for $\mathcal{E}(f,f)$.

To each $f \in \mathbb{D} \cap L^\infty$ there is associated a Radon
measure $\Gamma(f)$, called the \textbf{energy measure} of $f$,
defined by
\begin{equation}
  \int_X \phi \,d\Gamma(f) = 2 \mathcal{E}(\phi f, f) -
  \mathcal{E}(f^2, \phi)
\end{equation}
for $\phi \in C_c(X) \cap \mathbb{D}$, where $C_c(X)$ denotes the
space of continuous, compactly supported functions on $X$.  (It is
worth recalling that $\mathbb{D} \cap L^\infty$ is an algebra; see
\cite[Theorem 1.4.2 (ii)]{fukushima-et-al}.)  One may then define the
signed measure $\Gamma(f,g) = \frac{1}{2}(\Gamma(f+g) - \Gamma(f) -
\Gamma(g))$ by polarization.  For the classical Dirichlet form on
$\R^n$, we have $d\Gamma(f,g) = \grad f \cdot \grad g \,dm$.  We have
collected some useful properties of $\Gamma$ in Appendix
\ref{app-energy-measure}; further information can be found in
\cite{fukushima-classic}.

Let $L$ denote the self-adjoint generator of
$(\mathcal{E},\mathbb{D})$, with its domain $D(L)$.  We will take the
sign convention that $L$ is a \emph{positive} semidefinite operator,
so that $\mathcal{E}(f,g) = \inner{f}{Lg}_{L^2(X,\mu)}$ for $f \in
\mathbb{D}$, $g \in D(L)$.  Let $P_t$ denote the associated strongly
continuous, symmetric, Markovian semigroup.

Note that since $(\mathcal{E},\mathbb{D})$ is closed, $\mathbb{D}$ is
a Hilbert space under the inner product $\mathcal{E}_1(f,g) =
\mathcal{E}(f,g) + \int_X fg\,d\mu$.  (As before, $\mathcal{E}_1(f)$
is short for $\mathcal{E}_1(f,f)$.)  The inclusion $\mathbb{D}
\hookrightarrow L^2(X,\mu)$ is obviously continuous, 1-1, and has
dense image, so taking its adjoint gives another inclusion $L^2(X,\mu)
\hookrightarrow \mathbb{D}^*$ which is also continuous, 1-1, and has
dense image.  (So $(\mathbb{D}, L^2(X,\mu), \mathbb{D}^*)$ is a
Gelfand triple.)  Under these identifications, the pairing $(\ell,
f)_{\mathbb{D}^*, \mathbb{D}}$ is given by $\int_X \ell(x) f(x)
\mu(dx)$ when $\ell \in L^2(X,\mu) \subset \mathbb{D}^*$.  We will try
to keep denoting this pairing by $(\ell, f)_{\mathbb{D}^*,
  \mathbb{D}}$, but it would be a permissible abuse of notation to
just write $\int_X \ell(x) f(x) \mu(dx)$ in all cases.

In the classical case when $X = \R^n$, $\mu=m$ is Lebesgue measure and
$\mathcal{E}(f,g) = \frac{1}{2} \int_{\R^n} \grad f \cdot \grad g\,dm$
on the domain $\mathbb{D} = H^1(\R^n)$ (so $L = -\frac{1}{2}\Delta$)
we have $\mathbb{D}^* = H^{-1}(\R^n)$.  So it is helpful to think of
$\mathbb{D}^*$ as some space of distributions on $X$.

We note that the assumption that $(\mathcal{E},\mathbb{D})$ is regular
may impose some implicit ``boundary conditions'' on functions $f \in
\mathbb{D}$.  For example, if $X = \Omega$ is a bounded open subset of
$\R^n$ and $\mathcal{E}$ is the classical Dirichlet form, then
$\mathbb{D} = H^1_0(\Omega)$, so functions in the domain $\mathbb{D}$
must satisfy Dirichlet boundary conditions on $\partial \Omega$.





Intuitively, we want to consider solutions $u : (0,T) \times X \to \R$
of the heat equation
\begin{equation}
  \label{heat-eqn}
  \partial_t u + L u = 0.
\end{equation}
However, the implicit assumption in (\ref{heat-eqn}) that $u(t, \cdot)
\in D(L)$ for each $t$ is much too strong.  In particular, it is a
global condition: it requires that $u(t,\cdot)$ is in $L^2(X,\mu)$ and
satisfies certain boundary conditions.  We want something more
analogous to the classical heat equation
\begin{equation}
  \label{classical-heat-eqn}
  \partial_t u - \frac{1}{2} \Delta u = 0
\end{equation}
on an open set $\Omega \subset \R^n$, which is only a local equation.
It makes sense for any $u \in C^2$, for instance, and imposes no
global conditions such as integrability or behavior at the boundary.

The resulting notion of ``local weak solution'' has a more complicated
definition than (\ref{heat-eqn}), but comes closer to the intuition of
(\ref{classical-heat-eqn}) in that it avoids global conditions.

\begin{notation}
  Let $L^2([a,b]; \mathbb{D})$ denote the Hilbert space of (strongly
   measurable) square-integrable vector-valued functions $u : [a,b]
   \to \mathbb{D}$ under the norm
  \begin{equation*}
    \norm{u}_{L^2([a,b];\mathbb{D})}^2 := \int_a^b \mathcal{E}_1(u(t), u(t))\,dt.
  \end{equation*}
    $m$ will denote Lebesgue measure on the
  time interval $[a,b]$.  
\end{notation}

\begin{notation}
  Let $W^{1,2}([a,b]; \mathbb{D}, \mathbb{D}^*)$ denote the
  vector-valued Sobolev space of functions $u \in L^2([a,b];
  \mathbb{D})$ with one time derivative $u' \in L^2([a,b];
  \mathbb{D}^*)$.  We will write $W^{1,2}$ for short.  $W^{1,2}$ is a
  Hilbert space under the norm
  \begin{equation*}
    \norm{u}_{W^{1,2}([a,b]; \mathbb{D},\mathbb{D}^*)}^2 := \int_a^b
    (\mathcal{E}_1(u(t), u(t)) + \norm{u'(t)}_{\mathbb{D}^*}^2)\,dt.
  \end{equation*}
\end{notation}

A standard ``Sobolev embedding theorem'' gives that $W^{1,2}([a,b];
\mathbb{D}, \mathbb{D}^*) \subset C([a,b]; L^2(X,\mu))$, so that
$u(t)$ is well-defined as an element of $L^2(X,\mu)$ for every $t \in
[a,b]$, and $u$ is also well-defined as an element of $L^2([a,b]
\times X)$.  Also, we recall the following ``product rule'': if $u \in W^{1,2}$, then $t
\mapsto \norm{u(t)}_{L^2(X,\mu)}^2$ is an absolutely continuous
function on $[a,b]$, and
\begin{equation}\label{product-rule}
  \frac{d}{dt} \norm{u(t)}_{L^2(X,\mu)}^2 = \inner{u'(t)}{u(t)}_{\mathbb{D}^*,\mathbb{D}}.
\end{equation}
We refer the reader to \cite[\S 25]{wloka-book} for proofs.

One can show that any function $u :[a,b] \to \mathbb{D}$ which is in
$W^{1,2}$ is represented by an (almost everywhere defined) function
from $[a,b] \times X$ to $\R$, so we may write either $u(t)$ or
$u(t,x)$ depending on whether we prefer to think of $u$ as a curve in
a function space or a real-valued function of time and space.

\begin{definition}[See, e.g., \cite{sturm-dirichlet-ii,sturm-dirichlet-iii}]
\label{lws-def}
  Let $U \subset X$ be open.  A function $u : (0,T) \times U \to \R$
  is said to be a \textbf{local weak solution} (of the heat equation
  \ref{heat-eqn}) if for every compact $K \subset (0,T) \times U$,
  there exists a function $u_K \in W^{1,2}([0,T]; \mathbb{D},
  \mathbb{D}^*)$ such that $u = u_K$ ($m \times \mu$-a.e.) on $K$, and such that for every
  $\phi \in W^{1,2}([0,T]; \mathbb{D}, \mathbb{D}^*)$ with compact
  support inside $K$, we have
  \begin{equation}\label{lws-eq}
    \int_0^T (u_K'(t), \phi(t))_{\mathbb{D}^*, \mathbb{D}}\,dt +
    \int_0^T \mathcal{E}(u_K(t), \phi(t))\,dt = 0.
  \end{equation}
\end{definition}

Note that, because the Dirichlet form $(\mathcal{E},\mathbb{D})$ is
local, the expression in (\ref{lws-eq}) does not depend on the choice
of $u_K$.  Therefore, for test functions $\phi$ which are compactly
supported inside $(0,T) \times U$, we can interpret an expression like
$\int_0^T \mathcal{E}(u(t), \phi(t))\,dt$ as shorthand for $\int_0^T
\mathcal{E}(u_K(t), \phi(t))\,dt$ where $K \subset (0,T) \times U$ is
any compact set containing the support of $\phi$.  We shall henceforth
commit this abuse of notation.

We also note that the left side of (\ref{lws-eq}) is continuous in
$\phi$ with respect to the $L^2([0,T]; \mathbb{D})$ topology.
Therefore, since $W^{1,2}([0,T]; \mathbb{D},\mathbb{D}^*)$ is dense in
$L^2([0,T]; \mathbb{D})$, (\ref{lws-eq}) holds for all $\phi \in
L^2([0,T]; \mathbb{D})$ which have compact support inside $K$.

Observe that we do not assume that $u(t) \in \mathbb{D}$ or even $u(t)
\in L^2(X,\mu)$.  This definition gives no control on the behavior of
$u$ ``at infinity,'' i.e. away from compact sets, and $u$ need not
satisfy any implicit Dirichlet boundary conditions.  Each $u_K$ does
need to satisfy them, but $u$ can be totally different from $u_K$
outside the compact set $K$.  
There is also no assumption as to what happens near $t=0$ and $t=T$.

\section{Assumptions}\label{assumptions-sec}

In this section, we collect, and discuss, the hypotheses under which
we shall prove our results.

\begin{assumption}\label{dirichlet-space-ass}
  $(X,d)$ is a separable, locally compact, connected, locally connected metric
  space.  $\mu$ is a positive Radon measure on $X$ with full support.
  $(\mathcal{E},\mathbb{D})$ is a symmetric, regular, strictly local
  Dirichlet form on $L^2(X,\mu)$.
\end{assumption}


We shall now assume that local weak solutions satisfy a parabolic
Harnack inequality.  This inequality was developed by Moser, extending
previous work of Hadamard and Pini; see
\cite{moser64,moser64-correction} and references therein for further
history.


\begin{assumption}\label{harnack}
  Nonnegative local weak solutions satisfy the following
  \textbf{parabolic Harnack inequality}.  Let $U \subset X$ be open
  and connected, let $K \subset U$ be compact, and fix $0 < a < b
  < c < d < T$.  There exists a finite constant $C = C(U,K,a,b,c,d)$
  such that for every nonnegative local weak solution $u$ on $(0,T)
  \times U$,
  \begin{equation}
    \esssup_{[a,b] \times K} u \le C \essinf_{[c,d] \times K} u.
  \end{equation}
\end{assumption}

We make no assumption as to the exact dependence of $C$ on
$U,K,a,b,c,d$, so this Harnack inequality is local and in no way
scale-invariant.

In the context of local Dirichlet spaces, various forms of the
parabolic Harnack inequality are known to be related to other
functional inequalities, such as Poincar\'e inequalities and heat
kernel estimates.  See, for example,
\cite{sturm-dirichlet-ii,sturm-dirichlet-iii,bbk06}.  We shall discuss
this further in Section \ref{example-sec}.

\begin{assumption}\label{cts-heat-kernel}
  The semigroup $P_t$ admits a continuous heat kernel.  That is, there is a
  continuous $p : (0,\infty) \times X \times X \to \R$ such that, for
  all $f \in L^2(X,\mu)$,
  \begin{equation*}
    P_t f(x) = \int_X p(t,x,y) f(y) \mu(dy) \quad \text{for $\mu$-a.e
    $x \in X$}.
  \end{equation*}
\end{assumption}

Since the semigroup $P_t$ is Markovian, it follows that $p \ge 0$.  In
fact, since $u(t,x) = p(t, x_0, x)$ is a local weak solution, the
Harnack inequality (Assumption \ref{harnack}) implies that $p > 0$.
We will make use of this later.


\begin{notation}
  For any positive Radon measure $\nu$ on $X$, set
  \begin{equation}\label{Pt-nu-def}
    P_t \nu (x) = \int_X p(t,x,y) \nu(dy).
  \end{equation}
\end{notation}

Note that the continuity of $p$ is needed to ensure that $P_t \nu$ is
well-defined for measures $\nu$ which are not necessarily absolutely
continuous to $\mu$.

\begin{assumption}\label{Pt-cts-pointwise}
  For $f \in C_c(X)$ and every $x \in X$, we have $P_t f(x) \to f(x)$
  pointwise as $t \to 0$.  (Note that since the heat kernel $p$ is
  assumed to be continuous, $P_t f$ is continuous and hence $P_t f(x)$
  is well-defined for every $x \in X$.)
\end{assumption}

Assumption \ref{Pt-cts-pointwise} is very mild.  Indeed, we are not
aware of any examples which satisfy our previous assumptions but not
Assumption \ref{Pt-cts-pointwise}; however, we have not been able to
show that Assumption \ref{Pt-cts-pointwise} follows from them.

We emphasize that in Assumption \ref{Pt-cts-pointwise}, we require
that the convergence hold for \emph{every} $x \in X$.  Under our other
assumptions, the convergence will certainly hold for
\textbf{quasi-every} $x \in X$ (i.e. except on a set which, in the
terminology of \cite{fukushima-et-al}, is \textbf{zero capacity}, or
equivalently is \textbf{exceptional}).  The Dirichlet form
$(\mathcal{E},\mathbb{D})$ is associated with a continuous Hunt
process $X_t$ which is \textbf{normal} in the sense that
$\mathbb{P}_x(X_0 = x) = 1$.  So by dominated convergence, we have
$\mathbb{E}_x [f(X_t)] \to f(x)$ for every $x \in X$.  On the other
hand, it also holds that for quasi-every $x$, $P_t f(x) = \mathbb{E}_x
[f(X_t)]$ for every $t > 0$ and every $f \in C_c(X)$.

The Hunt process $X_t$ is not unique (one can change its behavior on
an exceptional set of initial points), and in many cases, there is a
canonical choice of $X_t$ for which $P_t f(x) = \mathbb{E}_x [f(X_t)]$
holds everywhere.  One could circumvent Assumption
\ref{Pt-cts-pointwise} by taking as the starting point a particular
continuous Hunt process $X_t$ with a continuous transition density
$p(t,x,y)$ and taking $(\mathcal{E},\mathbb{D})$ to be the
corresponding Dirichlet form.


It is also worth noting that Assumption \ref{Pt-cts-pointwise} holds
automatically in the common case that $P_t$ is a Feller semigroup,
i.e. a strongly continuous semigroup on $C_0(X)$.  In this case the
convergence holds not only pointwise but uniformly.  Another situation
where the assumption holds is when the state space $X$ is a
homogeneous space such as a Lie group, and the Dirichlet form is
translation invariant.

In Definition \ref{lws-def}, the existence of the ``localized''
functions $u_K$ is merely asserted, and we have little control over
them.  To produce them more constructively, we would like to use some
``smooth'' cutoff functions, as follows.

\begin{definition}\label{cutoff-def}
  Let $U$ be open and $K \subset U$ compact.  We will say a function
  $\psi : X \to [0,1]$ is a \textbf{nice cutoff function for $K$
  inside $U$} if:
  \begin{enumerate}
    \item $\psi = 1$ on $K$;\label{cutoff-1-on-K}
    \item $\psi$ is compactly supported inside $U$;\label{cutoff-inside-U}
    \item For every $f \in \mathbb{D}$, $\psi f \in \mathbb{D}$.  (It
    follows from the closed graph theorem that there is a constant
    $C_\psi$ such that $\mathcal{E}_1(\psi f) \le
    C_\psi \mathcal{E}_1(f)$.  And in particular, $\psi \in \mathbb{D}$.)
    \label{cutoff-preserve-domain}
  \end{enumerate}
\end{definition}

\begin{definition}\label{nice-def}
  We shall say an open set $U \subset X$ is \textbf{nice} if for every
  compact $K \subset U$ there exists a nice cutoff function for $K$
  inside $U$.
\end{definition}



\begin{assumption}\label{union-of-nice}
  There is a sequence of open, precompact, connected, nice sets $U_n \uparrow X$.
\end{assumption}

Some sufficient conditions for this assumption to hold are as follows:

\begin{enumerate}
\item If $X$ is compact, then $X$ itself is nice, since $\psi = 1$ is
  a nice cutoff function of any closed set inside $X$.  So we can take $U_n = X$.
  Actually, when $X$ is compact, several of the arguments in this
  paper become trivial.

\item If the intrinsic pseudo-distance defined by (\ref{intrinsic-d})
  below is a genuine metric and generates the topology of $X$, then
  every open set is nice: we can use the metric to construct cutoff
  functions that are Lipschitz.  This is a common assumption in the
  theory of Dirichlet spaces.  See Section \ref{sec-id} for more on
  this notion.

  This reduces Assumption \ref{union-of-nice} to the statement that
  we can exhaust $X$ by open, precompact, connected sets (which we
  shall abbreviate OPC for this paragraph).  But this follows from our
  topological assumptions on $X$.  Indeed, if we write $x \sim y$
  whenever there is an OPC set containing $x,y$, then by local
  compactness and local connectedness $\sim$ is an equivalence
  relation on $X$, and every equivalence class is open.  But then by
  connectedness of $X$ there is only one equivalence class.  Thus any
  pair of points, and moreover any finite set of points, is contained
  in an OPC set.  The same holds for any compact set $K$: we can cover
  $K$ by a finite number of OPC sets $U_n$.  Picking one $x_n$ from
  each $U_n$, we can find an additional OPC set $U_0$ which contains all the
  $x_n$.  Now the union of $U_0$ and the $U_n$ is an OPC set which
  contains $K$.  Since $X$ is $\sigma$-compact, this suffices.

\item If the Dirichlet space satisfies a local cutoff Sobolev
  inequality $\mathrm{CS}(\beta)_{\mathrm{loc}}$ in the sense of
  \cite{bbk06}, then all open sets are nice, and so Assumption
  \ref{union-of-nice} holds.  See Appendix \ref{cutoff-sobolev}.  Note
  that $\mathrm{CS}(\beta)_{\mathrm{loc}}$ is shown in \cite{bbk06} to
  follow from a uniform parabolic Harnack inequality.  The authors
  discuss examples of spaces, such as certain fractals, that are known
  to satisfy this condition, although they do not have a well-behaved
  intrinsic distance.  See also Section \ref{sec-fractal}.
\end{enumerate}

The useful consequence of Assumption \ref{union-of-nice} is that if
$U$ is nice and $u$ is a local weak solution on $(0,T) \times U$, then
for compact $K \subset (0,T) \times U$, we can produce the function
$u_K$ explicitly as follows.  We may enclose $K$ inside some set
$[a,b] \times K_1$ for a compact $K_1 \subset U$.  Let $\psi$ be a
nice cutoff function for $K_1$ inside $U$, and let $\chi \in
C^\infty_c((0,T))$ be a cutoff function which equals $1$ on $[a,b]$.
Set $K' = \supp \chi \times \supp \psi$ and let $u_K(t,x) = \chi(t)
\psi(x) u_{K'}(t,x)$.  This gives the same function $u_K$ no matter
which of the many possible choices for $u_{K'}$ is used, and it is not
hard to see that this $u_K$ is in $W^{1,2}([0,T]; \mathbb{D},
\mathbb{D}^*)$.

\section{Examples}\label{example-sec}

\if 
Examples of spaces satisfying the previous hypotheses.

General class of examples: volume doubling and local Poincar\'e
inequality.

Any non-examples?

We make no assumptions as to how $C$ depends on $U,K, a,b,c,d$.  A
stronger assumption would be that $(X, \mathcal{E},\mathbb{D})$ is a
\textbf{Harnack-type Dirichlet space} as considered in
\cite{gyrya-saloff-coste}, \cite{sturm-dirichlet-i}, in which a
\emph{uniform} parabolic Harnack inequality holds.  This property is
equivalent to having upper and lower heat kernel estimates, and to
having volume doubling and a Poincar\'e inequality.

Also sufficient to satisfy Assumption \ref{harnack} is the parabolic
Harnack inequality $\mathrm{PHI}(\Psi)$ considered in \cite{bbk06}, or
its local version.  Again, it is known that this inequality holds for
certain types of fractals.

\fi

This section discusses examples of spaces satisfying the assumptions listed in 
Section \ref{assumptions-sec}. As long 
as one is interested in local self-adjoint Markov operators defined over a 
locally compact metric space, it should be clear that these assumptions 
are quite mild and are satisfied for a great many 
interesting examples. We will work under the very basic Assumption 
\ref{dirichlet-space-ass} and focus on various contexts where the 
other assumptions listed in Section \ref{assumptions-sec} are satisfied. 

\subsection{Intrinsic distance} \label{sec-id}
Under Assumption \ref{dirichlet-space-ass},
the {\bf intrinsic distance} $d_{\mathcal E}$ is defined by
\begin{equation}\label{intrinsic-d}
d_{\mathcal E}(x,y)=\inf\{ f(y)-f(x): f\in  C_c(X)\cap
\mathbb D, d\Gamma(f)\le d\mu\},\;\;x,y\in X.\end{equation} Here
$d\Gamma(f)\le d\mu$ means that the Radon measure $\Gamma(f)$ is
absolutely continuous with respect to $\mu$ with Radon--Nikodym
derivative bounded by $1$.  The function $(x,y)\mapsto d_{\mathcal
E}(x,y)$ may be $0$ even if $x\neq y$ or may be infinity for some
$x,y$ but is symmetric and satisfies the triangle inequality. Loosely
speaking, the condition $d\Gamma(f)\le d\mu$ can be understood as
requiring that $f$ is ``Lipschitz'' with constant $1$. On a complete
Riemannian manifold equipped with its natural Dirichlet space
structure, $d_{\mathcal E}$ equals the Riemannian distance. 
See \cite{sturm-dirichlet-i, sturm-1998, biroli-mosco-saint-venant}.

Fractals such as the Sierpinski 
gasket and carpet provide examples of Dirichlet spaces with very 
interesting properties but where the intrinsic distance $d_\mathcal E$ is 
identically $0$ (the only functions with ``bounded gradient'' are the constant 
functions).  The infinite dimensional torus $\mathbb T^\infty$, 
the countable product of circles with its normalized Haar measure $\mu$,
equipped with  the Dirichlet form $\mathcal E_A $ associated with an infinite 
symmetric positive definite matrix $A=(a_{i,j})$,  
$$\mathcal E_A(f)=\int_{\mathbb T^\infty}
\sum_{i,j} a_{i,j}\partial_if \partial_jf d\mu,$$ 
provide examples where, depending on $A$,
the intrinsic distance may be finite and continuous or 
infinite except on a dense set of measure $0$. 
See \cite{bendikov-saloff-coste-diagonal}.

We define the balls $B(x,r)$ and the \textbf{volume growth function}
$V(x,r)$ by setting
$$B(x,r)=\{y: d_\mathcal E(x,y)<r\} \mbox{ and } V(x,r)=\mu (B(x,r)).$$
Please note that these balls 
are relative to the intrinsic distance $d_\mathcal E$.

Consider the following properties that may or may not be satisfied:
\begin{itemize}
\item[\textbf{(ID1)}] The intrinsic distance $d_\mathcal E$ is
  continuous and 
defines the topology of $X$.
\item[\textbf{(ID2)}] Property (ID1) is satisfied and for any compact set $K\subset X$
there are constants $r_K\in (0,\infty]$ and  $D_K,P_K\in (0,\infty)$ 
such that:
\begin{itemize}
\item (Compact balls) For all $x \in K$ and $r < r_K$, the closed ball
  $\overline{B(x, 2 r)}$ is compact.
\item (Doubling) For all any $x\in K$ and $r\in (0,r_K)$, we have $V(x,2r)\le D_KV(x,r)$.
\item (Poincar\'e inequality) For all any $x\in K$ and $r\in (0,r_K)$, 
$$\forall f\in \mathbb D,\; \int_{B(x,r)}|f-f_{B(x,r)}|^2d\mu\le P_K r^2
\int_{B(x,2r)}d\Gamma(f)$$
where $f_{B}$ denotes the mean of $f$ over $B$.
\end{itemize}
\end{itemize}
We say that \textbf{(ID2) holds locally uniformly} if there exists 
$r_0,D_0, P_0\in (0,\infty)$ such that (ID2) holds true with 
$r_{\{x\}}\ge r_0$, $D_{\{x\}}\le D_0 $ and $P_{\{x\}}\le P_0$ for all
$x \in X$.  Note that this implies that $\overline{B(x, r_0)}$ is
compact for every $x$, and in particular $d_\mathcal{E}$ is a complete metric on $X$.

We say that \textbf{(ID2) holds uniformly at all scales and positions} if there 
are constants $D_0$ and $P_0$ such that (ID2) holds with 
$r_K=\infty$, $D_K\le D_0$ and $P_K\le P_0$, for any compact $K\subset
X$.   Spaces where this holds are sometimes said to be of
\textbf{Harnack type}.  

As explained in \cite{sturm-dirichlet-i}, property (ID1) implies the
existence of an abundance of cutoff functions. In particular, under
(ID1), every open set is nice, and so Assumption \ref{union-of-nice}
is satisfied as argued in the previous section.

By the work of K.T. Sturm \cite{sturm-dirichlet-iii} extending to
Dirichlet spaces earlier work by Grigoryan \cite{grigoryan91} and
Saloff-Coste \cite{saloff-coste-psh}, the following theorem holds.

\begin{theorem}\label{th-H2}
Assuming {\em (ID1)}, property {\em (ID2)} holds if and only if
for any compact $K$ there exists a constant $H_K$ such that any nonnegative 
local weak solution $u$ of the heat equation {\em (\ref{heat-eqn})} in 
$Q= (0,r^2)\times B(x,r)$, $x\in K$, $r\in (0,r_K)$, satisfies
\begin{equation}
\esssup_{Q_-}\{u\}\le H_K \essinf_{Q_+}\{u\},
\end{equation}
where $Q_-=(r^2/4,r^2/2)\times B(x,r/2)$ and $Q_+=(3r^2/4,r^2)\times B(x,r/2)$.

Further, if {\em (ID2)} holds locally uniformly and $r_0$ is as
described above, then $r_{\{x\}}\ge r_0>0$ and 
there is a constant $H_0$  such that $H_{\{x\}}\le H_0$ for all $x\in X$. If 
{\em (ID2)} holds uniformly at all scales and locations then $r_K=\infty$
and there is a constant $H_0$ such that $H_K\le H_0$ for all 
compact $K\subset X$.
\end{theorem}

The parabolic Harnack inequality supplied by Theorem \ref{th-H2} can
be shown to imply that of Assumption \ref{harnack}, by covering the
set $K$ that appears in Assumption \ref{harnack} with finitely many
sufficiently small overlapping balls.  Moreover, it is argued in
\cite{sturm-dirichlet-ii} that under (ID1)--(ID2) a measurable heat
kernel exists.  Another very important consequence of (ID1)--(ID2), as
shown in \cite[Corollary 3.3]{sturm-dirichlet-iii}, is that local weak
solutions are locally H\"older continuous.  In particular the heat
kernel is continuous, so Assumption \ref{cts-heat-kernel} is also
satisfied.

Using the continuity of local weak solutions, we can also show that
Assumption \ref{Pt-cts-pointwise} is satisfied.  Fix $f \in C_c(X)$,
$x_0 \in X$ and $\epsilon > 0$.  Using regularity and the Markovian
properties of $(\mathcal{E},\mathbb{D})$, we may choose $g \in C_c(X)
\cap \mathbb{D}$ with $\norm{f-g}_\infty < \epsilon$ and such that $g$
is constant on a neighborhood $U$ of $x_0$.  Since $P_t f, P_t g$ are
continuous on $(0, \infty) \times X$ and $P_t$ is Markovian, we also
have $\norm{P_t f - P_t g}_\infty < \epsilon$.  Now it is shown in
\cite[Lemma 2.28]{gyrya-saloff-coste} that if we set
\begin{equation*}
  u(t,x) = 
  \begin{cases}
    P_t g(x), & t > 0 \\
    g(x_0), & t \le 0
  \end{cases}
\end{equation*}
then $u$ is a local weak solution on $(-\infty, \infty) \times U$.
(This also follows from our extension principle, Lemma
\ref{extension}, below.)  Therefore $u$ is continuous on $(-\infty,
\infty) \times U$, and in particular $P_t g \to g(x_0)$ uniformly on
compact subsets of $U$.  Letting $\epsilon \to 0$, it follows that
$P_t f \to f$ uniformly on compact subsets of $U$, and in particular
$P_t f(x_0) \to f(x_0)$.

To summarize:

\begin{corollary} 
Let $(X,d)$ be a separable, locally compact, connected, locally connected metric
space equipped  with a positive Radon measure $\mu$  with full support
and a symmetric, regular, strictly local
Dirichlet form   $(\mathcal{E},\mathbb{D})$ on $L^2(X,\mu)$. If 
$(X,d,\mu,\mathcal{E},\mathbb{D})$ satisfies {\em (ID1)--(ID2)} then
the assumptions 
{\em \ref{harnack}, \ref{cts-heat-kernel}, \ref{Pt-cts-pointwise}}
and {\em \ref{union-of-nice}} of {\em Section \ref{assumptions-sec}}  
are all satisfied.
\end{corollary}

Among the great many concrete examples of Dirichlet spaces satisfying
(ID1)--(ID2), let us mention inner uniform Euclidean domains equipped
with the canonical Neumann-type Dirichlet form
\cite{gyrya-saloff-coste}, complete Riemannian manifolds, connected
Lie groups equipped with an invariant sub-Riemannian structure given
by a generating family of left-invariant vector fields
\cite{purplebook}, the natural Dirichlet form on a polytopal complex
(under mild structural assumptions) \cite{eells-fuglede,pivarski-saloff-coste-complex}, and
Alexandrov spaces with sectional curvature bounded below
\cite{kuwae-et-al-alexandrov}.  Another example is the harmonic
Sierpi\'nski gasket described in \cite{kigami-measurable-riemannian},
which is something of a bridge to the fractal-like spaces of the next
subsection.

\subsection{Fractal-like spaces}\label{sec-fractal}
For the purposes of this paper, we call a Dirichlet space 
satisfying Assumption \ref{dirichlet-space-ass} \textbf{fractal-like of type $\beta$}
if the following local parabolic  
Harnack inequality relative to the metric balls $B_d(x,r)$ of $(X,d)$ 
is satisfied:
\begin{itemize}
\item[(H-$\beta$)]
For any compact $K$ there exist constants $r_K, H_K$ such that for any
$x \in K$, we have $\overline{B(x, 2 r_K)}$ is compact, and any nonnegative 
local weak solution $u$ of the heat equation (\ref{heat-eqn}) in 
$Q= (0,r^\beta)\times B(x,r)$, $x\in K$, $r\in (0,r_K)$, satisfies
\begin{equation}\label{Hbeta}
\esssup_{Q_-}\{u\}\le H_K \essinf_{Q_+}\{u\},
\end{equation}
where $Q_-=(r^\beta/4,r^\beta/2)\times B(x,r/2)$ and $Q_+=(3r^\beta/4,r^\beta)
\times B(x,r/2)$.
\end{itemize}
As in Section \ref{sec-id}, we can also say that
  {\boldmath{$(\text{\bf H-}\beta)$}} 
\textbf{holds locally uniformly}
if $H_{\{x\}}$, $r_{\{x\}}$ may be chosen independent of
$x$, and \textbf{uniformly at all scales and positions} if moreover
we can take $r_{\{x\}} \equiv \infty$.  

The parameter $\beta$ is known as the \textbf{walk dimension} and
describes the space-time scaling in the Dirichlet space.  
Thanks to the work of Barlow, Bass, Kumagai, and their collaborators
(see \cite{bbk06} and references therein), property (H-$\beta$) can be
characterized in a way that is similar in spirit to the statement
provided by Theorem \ref{th-H2}.  

The parabolic Harnack inequality (\ref{Hbeta}) is certainly stronger
than that of Assumption \ref{harnack}, and continuity of the heat
kernel (and other local weak solutions) follows as well, thus
verifying Assumptions \ref{cts-heat-kernel} and \ref{Pt-cts-pointwise}
as in Section \ref{sec-id}.  Condition (ID1) often fails to hold in
fractal-like spaces, but a replacement is supplied by so-called
CS($\beta$) cutoff Sobolev inequalities as described in \cite{bbk06}.
The latter condition is implied by (\ref{Hbeta}) \cite[Theorem
2.16]{bbk06}, and guarantees the existence of reasonable cutoff
functions; in particular, as we show in Appendix \ref{cutoff-sobolev},
it implies Assumption \ref{union-of-nice}.  Thus, a fractal-like space
in the sense introduced above satisfies all the assumptions introduced
in Section \ref{assumptions-sec}.

Examples of fractal-like spaces in this sense include the Sierpi\'nski
gasket \cite{bp88}, generalized Sierpi\'nski carpets \cite{bbk06}, and
Laakso spaces \cite{steinhurst-laakso-unique}.

\subsection{Locally compact but infinite dimensional examples}\label{sec-torus}
An interesting classes of examples comes from a symmetric 
Gaussian semigroup on the infinite 
dimensional torus $\mathbb T^\infty$ (the countable product of circles) and
$\mathbb R^k\times \mathbb T^\infty$. Each of these spaces is equipped with its 
Haar measure  $\mu$. Non-degenerate symmetric Gaussian convolution semigroups of 
measures, $(\mu^A_t)_{t\ge 0}$, 
are in one-to-one correspondence with symmetric positive definite matrices 
$A=(a_{i,j})$ in such a way that for any smooth function $\phi$ depending only on 
finitely many coordinates, we have
$$\lim_{t\rightarrow 0} \frac{1}{t}(\mu^A_t(\phi)-\phi(0))=
\sum_{i,j}a_{i,j}\partial_i\partial_j\phi(0).$$
Here the partial derivatives refer to the natural coordinate system
in the corresponding product space. The sum on the right-hand side has only 
finitely many non-zero terms because $\phi$ depends only on finitely many 
coordinates. To say that $A$ is positive definite is to say that for any vector
$\xi=(\xi_i)$ with finitely many non-zero coordinates,
$\langle A\xi,\xi\rangle=\sum a_{i,j}\xi_i\xi_j\ge 0$ and 
$\langle A\xi,\xi\rangle= 0$ if and only if $\xi=0$.

The family of measures $(\mu^A_t)_{t\ge 0}$ defines a symmetric Markov 
semigroup  $f\mapsto f*\mu^A_t$ with associated Dirichlet form
$$\mathcal{E}_A(f,g)=\int \sum_{i,j}a_{i,j}\partial_i f\partial_j g \,d\mu.$$
The domain of this form is the closure of the smooth compactly supported functions 
depending only on finitely many coordinates in the norm 
$\|f\|_2+\mathcal E_A (f,f)^{1/2}$ and can be described more explicitly, 
see \cite{bendikov-saloff-coste-diagonal}.

As examples of Dirichlet spaces on locally compact spaces, these examples are 
interesting because of the great variety of very different behaviors. 
For instance, depending on the matrix $A$, the measures $\mu^A_t$ may or may 
not have a density with respect to the Haar measure $\mu$ and, if it exists, 
this density may or may not be continuous. Further, depending on $A$,
the intrinsic distance defined by (\ref{intrinsic-d}) may or may not 
have property (ID1) whereas property (ID2) is never satisfied.

The following theorem addresses the question of whether or not the assumptions of 
Section \ref{assumptions-sec} hold. We note that $\mathbb R^k\times 
\mathbb T^\infty$ is locally compact, metrizable, path connected 
and locally path connected. The Dirichlet forms $\mathcal{E}_A$ 
described above are regular and strictly local, so Assumption 
\ref{dirichlet-space-ass} is always satisfied. 
Assumption \ref{Pt-cts-pointwise} is also always satisfied (indeed, in this
case, if $f$ is continuous and compactly supported, $f*\mu^A_t(x)=\mathbb{E}_x[f(X_t)]$
everywhere, by invariance). Assumption \ref{union-of-nice} is also always satisfied (use 
smooth cutoff functions that depend only on finitely many coordinates).

Define 
$$W_A(s)=\#\{\theta\in \mathbb Z^{(\infty)}: 
\langle A \theta,\theta\rangle \le s\}$$
where $\mathbb Z^{(\infty)}$ is the set of integer valued sequences with 
finitely many non-zero entries. The function $W_A$ may be infinite for some $s$.
\begin{theorem}[\cite{bendikov-saloff-coste-diagonal}] Referring to the above setting and notation, we have:
\begin{itemize}
\item {\em Assumption \ref{cts-heat-kernel}} is satisfied if and only if
$$\lim_{s\rightarrow \infty}\frac{1}{s}\log W_A(s)=0.$$
\item {\em Assumption \ref{harnack}} is satisfied if and only if
$$\lim_{s\rightarrow \infty}\frac{1}{\sqrt{s}}\log W_A(s)=0.$$
This is also equivalent to the statement that $\mu^A_t$ is absolutely continuous
with respect to the Haar measure $\mu$ and admits a continuous density 
$x\mapsto \mu^A_t(x)$ such that 
$$\lim_{t\rightarrow 0} t\log \mu^A_t(0)=0.$$  
\end{itemize}
\end{theorem}

Computing the functions $W_A$ is a very difficult task. However, 
the results become much more explicit in the case when $A$ is a diagonal matrix 
with diagonal entries $a_{i,i}=a_i>0$. In this case, set
$$N_A(s)=\#\{ i: a_i\le s\}.$$ 
Then, Assumption \ref{cts-heat-kernel} is satisfied if and only if 
$$\lim_{s\rightarrow \infty} \frac{1}{s}\log N_A(s)=0.$$
Assumption \ref{harnack} is satisfied if and only if
$$\lim_{s\rightarrow \infty} \frac{1}{s} N_A(s)=0.$$

\section{Properties of local weak solutions}

In this section, we collect a number of facts about local weak
solutions that we will use in the proof of the main theorem.

We can integrate by parts in (\ref{lws-eq}) and put the time
derivative on the  test function $\phi$.  This results in a statement
that makes sense for $u$ which are not \emph{a priori} assumed to have
an $L^2$ time derivative.  By the following lemma, this new statement
is equivalent to Definition \ref{lws-def}.

\begin{lemma}\label{weaker-solution}
  Let $U \subset X$ be open and nice.  A function $u \in L^2([0,T];
  \mathbb{D})$ is a local weak solution on $(0,T) \times U$ if and
  only if it satisfies
  \begin{equation}\label{weaker-solution-eqn}
    -\int_0^T (\phi'(t), u(t))_{\mathbb{D}^*, \mathbb{D}}\,dt +
     \int_0^T \mathcal{E}(\phi(t), u(t))\,dt = 0
  \end{equation}
  for all $\phi \in W^{1,2}([0,T]; \mathbb{D}, \mathbb{D}^*)$ which
  are compactly supported inside $(0,T) \times U$.
\end{lemma}

\begin{proof}
  The forward implication is trivial, using integration by parts.  For
  the converse, let $V$ be open and precompact with $\closure{V}
  \subset U$.  Let $\psi$ be a nice cutoff function for $\closure{V}$
  inside $U$.  Then $\phi
  \mapsto - \int_0^T \mathcal{E}(u(t), \psi \phi(t))\,dt$ is a bounded
  linear functional on $L^2([0,T]; \mathbb{D})$.  But $L^2([0,T];
  \mathbb{D})^* = L^2([0,T]; \mathbb{D}^*)$ \cite[Theorem
  II.13.5.8, Corollary 1]{dinculeanu}, so there exists $w \in L^2([0,T];
  \mathbb{D}^*)$ such that for all $\phi \in L^2([0,T]; \mathbb{D})$,
  \begin{equation*}
    \int_0^T (w(t), \phi(t))_{\mathbb{D}^*, \mathbb{D}}\,dt =
    -\int_0^T \mathcal{E}(u(t), \psi \phi(t))\,dt.
  \end{equation*}

  For any $\chi \in C_c^\infty((0,T))$ and any $f \in \mathbb{D}$, we
  note that $\chi(t)\psi f \in W^{1,2}([0,T]; \mathbb{D},
  \mathbb{D}^*)$ with compact support inside $(0,T) \times U$.  Thus
  we have
  \begin{align*}
    \left(\int_0^T w(t) \chi(t)\,dt, f\right)_{\mathbb{D}^*,
    \mathbb{D}}
    &= \int_0^T (w(t), \chi(t) f)_{\mathbb{D}^*, \mathbb{D}}\,dt \\ 
    &= -\int_0^T \mathcal{E}(u(t), \chi(t) \psi f)\,dt \\
    &= -\int_0^T (\chi'(t) \psi f, u(t))_{L^2}\,dt
    \\
    &= -\left(\int_0^T \chi'(t) \psi u(t)\,dt, f\right)_{L^2}.
  \end{align*}
  Thus $w$ is the weak derivative of $\psi u$, and so we have $\psi u
  \in W^{1,2}$.

  Now if $\phi \in W^{1,2}$ with compact support inside $(0,T) \times
  V$, we have
  \begin{align*}
    \int_0^T (\psi u'(t), \phi(t))_{\mathbb{D}^*, \mathbb{D}}\,dt &=
    -\int_0^T \mathcal{E}(u(t), \psi \phi(t))\,dt \\
    &= -\int_0^T \mathcal{E}(\psi u(t), \phi(t))\,dt
  \end{align*}
  since $\mathcal{E}$ is local and $\psi = 1$ on the support of
  $\phi$.  Thus we have produced a function (namely $\psi u$) which equals
  $u$ a.e. on $(0,T) \times V$ and satisfies the necessary equation.
  Since $V$ was arbitrary, $u$ is a local weak solution on $(0,T)
  \times U$.
\end{proof}

We now give an ``extension principle'' giving conditions for a local
weak solution to have an extension backwards in time.  In the setting
of Euclidean space, a similar result was given in \cite{aronson68}.

\begin{lemma}[Extension principle]  \label{extension}
  Let $U$ be open and nice, and let $u$ be a local weak solution on
  $(0,T) \times U$.  Suppose that for any nice cutoff
  function $\psi$ compactly supported inside $U$, we have $\psi u \in
  L^2([0,T]; \mathbb{D})$, and $\psi u(t) \to 0$ weakly in
  $L^2(X,\mu)$ as $t \downarrow 0$.  Extend $u$ by setting $u(t) = 0$
  for $t \le 0$. 
  Then $u$ is a local weak solution on $(-\infty, T) \times U$.  
\end{lemma}

\begin{proof}
  Let $V$ be open and precompact with $\closure{V} \subset U$, and let
  $\psi$ be a nice cutoff function of $\closure{V}$ inside $U$.  Let $\phi \in
  W^{1,2}((-\infty; T]; \mathbb{D}, \mathbb{D}^*)$ be compactly supported
  inside $(-\infty, T) \times V$.  Fix any $\epsilon > 0$, and let
  $\chi_n \in C_c((\epsilon, T))$ be a bounded sequence of cutoff functions
  with $\chi_n \to 1_{[\epsilon, T)}$ pointwise.  
  Integrating by parts,
  we have
  \begin{align*}
    \int_{\epsilon}^T (\phi'(t), \psi u(t))_{\mathbb{D}^*, \mathbb{D}}\,dt
    &= - (\phi(\epsilon), \psi u(\epsilon))_{L^2(X,\mu)} -
    \int_{\epsilon}^T (\psi u'(t), \phi(t))_{\mathbb{D}^*, \mathbb{D}}\,dt \\
    &= - (\phi(\epsilon), \psi u(\epsilon))_{L^2(X,\mu)} -
    \lim_{n \to \infty} \int_{\epsilon}^T 
    (\psi u'(t), \chi_n(t) \phi(t))_{\mathbb{D}^*, \mathbb{D}}\,dt \\
    &= - (\phi(\epsilon), \psi u(\epsilon))_{L^2(X,\mu)} +
    \lim_{n \to \infty} \int_{\epsilon}^T \mathcal{E}(\psi u(t),
    \chi_n(t) \phi(t))\,dt \\
    &= - (\phi(\epsilon), \psi u(\epsilon))_{L^2(X,\mu)} +
    \int_{\epsilon}^T \mathcal{E}(\psi u(t),
    \phi(t))\,dt,
  \end{align*}
  since $u$, and hence $\psi u$, is a local weak solution on $(0,T)
  \times V$.  The limits involving $\chi_n$ converge as desired by
  dominated convergence.

  Now let $\epsilon \to 0$.  Since $\phi \in W^{1,2}([0,T];
  \mathbb{D}, \mathbb{D}^*)$ and $\psi u \in L^2([0,T]; \mathbb{D})$,
  we have by dominated convergence that
  \begin{align*}
    \int_{\epsilon}^T (\phi'(t), \psi u(t))_{\mathbb{D}^*, \mathbb{D}}\,dt 
    &\to
    \int_0^T (\phi'(t), \psi u(t))_{\mathbb{D}^*, \mathbb{D}}\,dt 
    = 
    \int_{-\infty}^T (\phi'(t), \psi u(t))_{\mathbb{D}^*, \mathbb{D}}\,dt \\
    \int_{\epsilon}^T \mathcal{E}(\psi u(t), \phi(t))\,dt 
    &\to
    \int_{0}^T \mathcal{E}(\psi u(t), \phi(t))\,dt 
    = 
    \int_{-\infty}^T \mathcal{E}(\psi u(t), \phi(t))\,dt.
  \end{align*}
  Also, since $\phi \in C((-\infty, T]; L^2(X,\mu))$ we have
  $\phi(\epsilon) \to \phi(0)$ in $L^2(X,\mu)$, and $\psi u(\epsilon)
  \to 0$ weakly in $L^2(X,\mu)$, so $(\phi(\epsilon), \psi
  u(\epsilon))_{L^2(X,\mu)} \to 0$.  Thus we have shown
  \begin{equation*}
    \int_{-\infty}^T (\phi'(t), \psi u(t))_{\mathbb{D}^*,
    \mathbb{D}}\,dt = \int_{-\infty}^T \mathcal{E}(\phi(t), \psi u(t))\,dt
  \end{equation*}
  so by Lemma \ref{weaker-solution} we have that $\psi u$ (hence
  $u$) is a local weak solution on $(-\infty, T) \times V$.  Since $V
  \subset U$ was arbitrary, we are done.
\end{proof}

An important property of a local weak solution is that, locally, the
size of the function controls its energy.  This is the content of the
following inequality.
A similar inequality was used in \cite[Section 5.2.2]{saloff-coste-sobolev}.

\begin{lemma}[Caccioppoli-type inequality]
  \label{caccioppoli}
  Let $U$ be open, nice, and precompact, and let $V$ be open and
  precompact with $\closure{V} \subset U$.  Let $\psi$ be a nice
  cutoff function of $\closure{V}$ inside $U$.  There is a constant
  $C$, depending on $U,V,\psi,T$, such that, for every nonnegative
  local weak solution $u$ on $(0,T) \times U$, we have
  \begin{equation}
    \norm{\psi u}_{L^2([0,T]; \mathbb{D})} \le C \esssup_{[0,T] \times
    U} u.
  \end{equation}
  (Recall that $\norm{v}_{L^2([0,T]; \mathbb{D})}^2 = \int_0^T
  \mathcal{E}_1(v(t), v(t))\,dt = \int_0^T \left[ \norm{v(t)}_{L^2(X,\mu)}^2 +
  \mathcal{E}(v(t))\right]\,dt$.)
\end{lemma}

\begin{proof}
  Set $M = \esssup_{[0,T] \times U} u$.  Since $\int_0^T \norm{\psi
  u(t)}_{L^2(X,\mu)}^2\,dt \le T M^2$, it will be enough to show $\int_0^T
  \mathcal{E}(\psi u(t))\,dt \le CM^2$.  

  By Proposition
  \ref{energy-product-bound}, we have
  \begin{equation}
    \mathcal{E}(\psi u(t), \psi u(t)) \le 2 \int_X u(t)^2
    \,d\Gamma(\psi) + 2\int_X \psi^2 \,d\Gamma(u(t)).
  \end{equation}
  The first term is bounded by $2 M^2 \mathcal{E}(\psi)$ so we
  work on the second term.  By repeated
  application of the product rule (Proposition \ref{energy-leibniz}),
  \begin{align*}
    \int_X \psi^2 \,d\Gamma(u(t)) &= \mathcal{E}(u(t), \psi^2 u(t)) - 2 \int_X u(t) \psi \,d\Gamma(u(t),
    \psi) \\
    &\le \mathcal{E}(u(t), \psi^2 u(t))  + \abs{ \int_X (2u(t)) \psi \,d\Gamma(u(t),
      \psi)} \\
    &\le \mathcal{E}(u(t), \psi^2 u(t)) + 2 \int_X u(t)^2 \,d\Gamma(\psi) +
    \frac{1}{2} \int_X \psi^2\,d\Gamma(u(t)) && \text{by (\ref{energy-amgm}).}
  \end{align*}
  Thus
  \begin{equation*}
    \int_X \psi^2 \,d\Gamma(u(t)) \le 2 \mathcal{E}(u(t), \psi^2 u(t)) + 4
    \int_X u(t)^2 \,d\Gamma(\psi) \le 2 \mathcal{E}(u(t), \psi^2 u(t)) + 4 M^2
    \mathcal{E}(\psi).
  \end{equation*}

  Now let $\chi_n \in C_c^\infty((0,T))$ with $\chi_n \uparrow 1$
  pointwise and $\int_0^T |\chi_n'(t)|\,dt \le 5$.  By monotone
  convergence, 
  \begin{align*}
    \int_0^T \int_X \psi^2 \,d\Gamma(u(t))\,dt &= \lim_{n \to
    \infty} \int_0^T \chi_n(t) \int_X \psi^2 \,d\Gamma(u(t))\,dt \\
    &\le 4 M^2 T
    \mathcal{E}(\psi) + \limsup_{n \to \infty} \int_0^T
    \mathcal{E}(u(t), \chi_n(t) \psi^2 u(t))\,dt.
  \end{align*}
  However, $\phi(t,x) = \chi_n(t) \psi(x)^2 u(t,x)$ is in
  $W^{1,2}([0,T]; \mathbb{D}, \mathbb{D}^*)$ and is compactly supported
  inside $(0,T) \times U$.  Thus, since $u$ is a local weak
  solution on $(0,T) \times U$, we have
  \begin{align*}
    \int_0^T
    \mathcal{E}(u(t), \chi_n(t) \psi^2 u(t))\,dt &= -\int_0^T
    (u'(t), \chi_n(t) \psi^2 u(t))_{\mathbb{D}^*, \mathbb{D}}\,dt \\
    &= -\int_0^T \chi_n(t) ((\psi u)'(t), \psi
    u(t))_{\mathbb{D}^*, \mathbb{D}}\,dt \\
    &= \int_0^T \chi_n'(t) ||\psi u(t)||_{L^2}^2\,dt \\
    & \le 5 M^2 \mu(U) 
  \end{align*}
  where in the next-to-last line we used (\ref{product-rule}) and integration by parts.  
This completes the proof.
\end{proof}

Combining Lemmas \ref{weaker-solution} and \ref{caccioppoli} lets us
show that a bounded limit of local weak solutions is another local
weak solution.

\begin{lemma}\label{weak-soln-limit}\label{convergence}
  Let $U$ be open and nice, and let $u_n$ be a sequence of nonnegative
  local weak solutions on $(0,T) \times U$ which are uniformly
  bounded, i.e. $0 \le u_n \le M$ on $(0,T) \times U$.  Suppose $u_n
  \to u$ pointwise.  Then $u$ is a local weak solution on $(0,T)
  \times U$.
\end{lemma}

\begin{proof}
  Let $V$ be open and precompact with $\closure{V} \subset U$, and let
  $\psi$ be a nice cutoff function of $\closure{V}$ inside $U$.  By
  Lemma \ref{caccioppoli} we have $\norm{\psi u_n}_{L^2([0,T];
  \mathbb{D})} \le C M$ for all $n$.  Passing to a subsequence, we may
  assume $\psi u_n$ converges weakly to $\psi u$ in the Hilbert space
  $L^2([0,T]; \mathbb{D})$.  By Lemma \ref{weaker-solution}, it
  immediately follows that $\psi u$ is a local weak solution on $(0,T)
  \times V$ (the left side of (\ref{weaker-solution-eqn}), as a
  function of $u$, is a continuous linear functional on $L^2([0,T];
  \mathbb{D})$).  Since $V \subset U$ was arbitrary, we are done.
\end{proof}

We now record some properties of local weak solutions produced by the
heat semigroup $P_t$.

\begin{lemma}\label{Pt-lws}
  Let $f \in L^2(X,\mu)$, and for each $t > 0$ let $u(t) = P_t f$.
  Then $u(t) \in W^{1,2}([0,T]; \mathbb{D}, \mathbb{D}^*)$ for any
  $T$; in fact, $u(t) \in D(L)$ for every $t > 0$, and $u'(t) =
  Lu(t)$.  In particular, $u$ is a local weak solution on $(0, \infty)
  \times X$.
\end{lemma}

\begin{proof}
  This is a simple consequence of the spectral theorem.
\end{proof}

For $\nu_n$ a sequence of finite positive Radon measures on $X$,
recall that $\nu_n \to \nu$ \textbf{weakly} if $\int_X f\,d\nu_n \to
\int_X f\,d\nu$ for every $f \in C_0(X)$, where $C_0(X)$ is the space
of continuous functions on $X$ which vanish at infinity (i.e. the
uniform closure of $C_c(X)$).  (In other words, this is weak-*
convergence in $C_0(X)^*$.)  Equivalently, $\nu_n \to \nu$ weakly iff
$\{\nu_n\}$ is bounded in total variation (i.e. $\sup_n \nu_n(X) <
\infty$) and $\int_X f\,d\nu_n \to \int_X f\,d\nu$ for every $f \in
C_c(X)$.  For a nonnegative measurable function $f$, we identify $f$
with the measure $f \,d\mu$, and note that the total variation of this
measure is $\norm{f}_{L^1(X,\mu)}$.

\begin{lemma}\label{Pt-weak-lim-zero}
  If $\nu$ is a compactly supported positive Radon measure, then $P_t \nu \to
  \nu$ weakly as $t \to 0$.  ($P_t \nu$ is as defined in (\ref{Pt-nu-def}).)
\end{lemma}

\begin{proof}
  It follows from the Markovian property of $P_t$ that $\norm{P_t
  \nu}_{L^1(X,\mu)} \le \nu(X)$, so $\{P_t \nu\}$ is bounded in total variation.
  
  For $f \in C_c(X)$, we have by Fubini's theorem and the symmetry of
  $P_t$ that $\int_X f(x) P_t \nu(x) \mu(dx) = \int_X P_t f(y)
  \nu(dy)$.  By Assumption \ref{Pt-cts-pointwise}, $P_t f \to f$
  pointwise, and the Markovian property gives $\norm{P_t f}_\infty \le
  \norm{f}_\infty < \infty$.  Thus by dominated convergence, $\int_X
  f(x) P_t \nu(x) \mu(dx) \to \int_X f(x) \nu(dx)$.
\end{proof}

Note that Assumption \ref{Pt-cts-pointwise} asserts that $P_t f(x) \to
f(x)$ for \emph{every} $x \in X$.  ``$\mu$-almost every $x$'' would not be
sufficient to establish Lemma \ref{Pt-weak-lim-zero}, as the measure $\nu$
could charge $\mu$-null sets.

\begin{lemma}\label{Pt-nu-solution}
  If $\nu$ is a compactly supported positive Radon measure, then $u(t,x) = P_t
  \nu(x)$ is a local weak solution on $(0,\infty) \times X$, and
  moreover $P_t \nu \in \mathbb{D}$ for each $t > 0$.
\end{lemma}

\begin{proof}
  It is sufficient to show that $P_t \nu \in L^2(X, \mu)$ for any $t >
  0$, since then we can write $P_t \nu = P_{t-\epsilon} P_{\epsilon}
  \nu$ and apply Lemma \ref{Pt-lws} (replacing $t$ by $t-\epsilon$) to
  get that $u$ is a local weak solution on $(\epsilon,\infty) \times
  X$, for arbitrary $\epsilon$.  If $\nu$ is supported on the compact
  set $K$, we have
  \begin{align*}
    \int_X |P_t \nu(x)|^2 \mu(dx) &= \int_X \abs{\int_K p(t,x,y)
    \nu(dy)}^2 \mu(dx) \\
    &\le \nu(K) \int_X \int_K p(t,x,y)^2 \nu(dy) \mu(dx) &&
    \text{(Cauchy--Schwarz)} \\
    &= \nu(K) \int_K \int_X p(t,x,y) p(t,y,x) \mu(dx) \nu(dy) &&
    \text{(by symmetry of $p$)} \\
    &= \nu(K) \int_K p(2t,y,y) \nu(dy) \\
    &\le \nu(K)^2 \sup_{y \in K} p(2t, y,y) < \infty
  \end{align*}
  since $p$ is continuous.
\end{proof}

\begin{lemma}\label{Pt-weak-cts}
  Let $\nu_n$ be a sequence of positive Radon measures supported in a single
  compact set $K \subset X$.  Suppose $\nu_n \to \nu$ weakly, and $t_n
  \to t \in (0,\infty)$.  Then $P_{t_n} \nu_n (x) \to
  P_t \nu(x)$ for each $x \in X$.
\end{lemma}

\begin{proof}
  Fix $x \in X$ and $\epsilon > 0$.  Write
  \begin{equation*}
    P_{t_n} \nu_n(x) - P_t \nu(x) = (P_{t_n} \nu_n(x) - P_t \nu_n(x))
    + (P_t \nu_n(x) - P_t \nu(x)).
  \end{equation*}
  For the first term, we have
  \begin{equation*}
    \abs{P_{t_n} \nu_n(x) - P_t \nu_n(x)} = \abs{\int_K
    (p(t_n,x,y)-p(t,x,y)) \nu_n(dy)} \le \nu_n(K) \sup_{y \in K} \abs{p(t_n,x,y)-p(t,x,y)}.
  \end{equation*}
  Since $\nu_n$ converges weakly, we have $\sup_n \nu_n(K) < \infty$.  And since $p$ is continuous, we
  have $p(t_n, x, \cdot) \to p(t,x,\cdot)$ uniformly on $K$ as $t_n
  \to t$.  So this term goes to zero.  The second term goes to zero by
  definition of weak convergence, since $p(t,x,\cdot)$ is a continuous
  function on $K$.
\end{proof}




An important fact about the heat semigroup is that $P_t f$ is the smallest of all nonnegative
local weak solutions which equal $f$ at time $t=0$.  Intuitively, $P_t
f$ is the solution which imposes Dirichlet conditions at the boundary
of $X$, so that heat flows out of $X$ as much as possible, and no heat
flows in.

\begin{proposition}\label{minimal}
  Let $u$ be a nonnegative local weak solution on $(-a,T) \times X$
  for some $-a < 0$,
  and suppose $f \in L^2(X, \mu)$ satisfies $f \le u(0)$ a.e.  Then
  $P_t f \le u$ a.e. on $[0, T) \times X$.
\end{proposition}

The proof requires considering what happens when we restrict our
attention to some open subset $U$ and impose Dirichlet boundary
conditions on $\partial U$.  We make use of the following results
which can be found in \cite{grigoryan-hu-off-diagonal}.  

\begin{definition}
  For $U \subset X$ open, let $\mathbb{D}(U) \subset \mathbb{D}$
  denote the $\mathcal{E}_1$-closure of $\mathbb{D} \cap C_c(U)$.
\end{definition}

Morally these are the functions from $\mathbb{D}$ satisfying Dirichlet
boundary conditions on $\partial U$.  There are several other possible
equivalent definitions.  Note in particular that if $f \in
\mathbb{D}(U)$, then $f = 0$ $\mu$-a.e. on $U^c$.  (This follows
because $f$ is an $\mathcal{E}_1$-limit of functions $f_n \in C_c(U)$,
and so a subsequence converges to $f$ almost everywhere.  In fact,
this can be upgraded to quasi-everywhere convergence.)

\begin{lemma}\label{control-bdy-cond}
  Suppose $U \subset X$ is open, and we have $f \in \mathbb{D}$, $g \in
  \mathbb{D}(U)$ with $0 \le f \le g$ a.e.  Then $f \in \mathbb{D}(U)$.
\end{lemma}

\begin{proof}
See Lemma 4.4 of \cite{grigoryan-hu-off-diagonal}.
\end{proof}


The following fact can easily be verified:

\begin{proposition}
  The restriction $(\mathcal{E}, \mathbb{D}(U))$ of $\mathcal{E}$ to
  $\mathbb{D}(U)$ defines a regular, strictly local, symmetric Dirichlet form on
  $L^2(U, \mu) \subset L^2(X, \mu)$.
\end{proposition}

\begin{proposition}
  Suppose $U \subset X$ is open and nice.  If $u$ is a local weak
  solution on $(0, T) \times U$ with respect to $(\mathcal{E},
  \mathbb{D})$, then it is also a local weak solution on $(0,T) \times
  U$ with respect to $(\mathcal{E}, \mathbb{D}(U))$.
\end{proposition}

\begin{proof}
  Let $K \subset (0,T) \times U$ be compact; without loss of
  generality we can take $K = [a,b] \times K_1$.  Then there exists
  $u_K \in W^{1,2}([0,T]; \mathbb{D}, \mathbb{D}^*)$ with $u_K = u$
  a.e. on $K$.  Set $\tilde{u}_K = \psi u_K$, where $\psi$ is a nice
  cutoff function of $K_1$ inside $U$.  Then it is easy to see that
  $\tilde{u}_K \in W^{1,2}([0,T]; \mathbb{D}(U), \mathbb{D}(U)^*)$.
  If $\phi \in W^{1,2}([0,T]; \mathbb{D}(U), \mathbb{D}(U)^*)$ with
  compact support inside $K$, then we also have $\phi \in
  W^{1,2}([0,T]; \mathbb{D}, \mathbb{D}^*)$, and so (\ref{lws-eq})
  holds.
\end{proof}

$\mathbb{D}(U)$ gives us a notion of ``functions vanishing on
$\partial U$'', which allows one to state the following parabolic
maximum principle.  A similar statement is proved in \cite[Proposition
4.11]{grigoryan-hu-off-diagonal}.

\begin{theorem}
  \label{max-principle}
  Let $U \subset X$ be open.  Suppose $u \in W^{1,2}([0,T];
  \mathbb{D}, \mathbb{D}^*)$ satisfies:
  \begin{enumerate}
    \item $u(0) \le 0$ a.e. (recall $u \in C([0,T]; L^2(X,\mu))$ so
    $u(0)$ is well defined as an $L^2(X,\mu)$ function);
    \item $u(t)^+ \in \mathbb{D}(U)$ for a.e. $t$;
    \item \label{max-test-function} For every $\phi \in L^2([0,T]; \mathbb{D}(U))$ which
    vanishes outside some $(\delta, T-\delta)$, we have
    \begin{equation}\label{u-max-weak}
      \int_0^T \left[(u'(t), \phi(t))_{\mathbb{D}^*, \mathbb{D}} +
      \mathcal{E}(u(t), \phi(t))\right]\,dt = 0.
    \end{equation}
  \end{enumerate}
  Then $u \le 0$ a.e. on $[0,T] \times U$.
\end{theorem}

The proof is a fairly straightforward adaptation of the argument in
\cite{grigoryan-hu-off-diagonal} and is relegated to Appendix
\ref{app-max-principle}.

Let $P_t^U$ be the semigroup generated by $(\mathcal{E},
\mathbb{D}(U))$.  Technically it is only a semigroup on $L^2(U, \mu)$,
but it can be extended to $L^2(X, \mu)$ in the obvious way (by
defining $P_t^U f = P_t^U (f|_U)$).  It is strongly continuous only on
$L^2(U, \mu)$.  By the spectral theorem we have $P_t^U f \in
\mathbb{D}(U)$ for all $f \in L^2(X, \mu)$.

Regularity says that the semigroup $P_t$ (and hence the Dirichlet
form) is determined by its behavior on sufficiently large open sets,
as made precise in the following lemma.

\begin{lemma}[{\cite[Lemma 4.17]{grigoryan-hu-off-diagonal}}]\label{exhaustion-semigroup}
  If $\{U_i\}_{i=1}^\infty$ is an increasing sequence of open subsets
  of $X$ with $X = \bigcup_{i=1}^\infty U_i$, then for any $t > 0$ and $0 \le f \in L^2(X, \mu)$, we have
  $P_t^{U_i} f \to P_t f$ $\mu$-a.e.
\end{lemma}

Combining these, we may prove Proposition \ref{minimal}.

\begin{proof}[Proof of Proposition \ref{minimal}]
  Let $u$ be a nonnegative local weak solution on $(-a,T) \times
  X$, and suppose $f \in L^2(X, \mu)$ satisfies $f \le u(0)$ a.e.
  Since $P_t f \le P_t f_+$, we can assume without loss of generality
  that $f \ge 0$ a.e.  (We shall only apply the theorem with
  nonnegative $f$ anyway.)

  Let $U$ be an open precompact set, and let $T' < T$ be arbitrary.
  Choose a nonnegative function $u_K \in W^{1,2}([0,T']; \mathbb{D},
  \mathbb{D}^*)$ which agrees with $u$ on some compact neighborhood of
  $[0, T'] \times \closure{U}$ and set $v(t) = P^U_t f -
  u_K(t) \in W^{1,2}([0,T']; \mathbb{D}, \mathbb{D}^*)$ (because
  $W^{1,2}([0,T']; \mathbb{D}(U), \mathbb{D}(U)^*) \subset
  W^{1,2}([0,T']; \mathbb{D}, \mathbb{D}^*)$).
  We claim $v$ satisfies the hypotheses of Theorem
  \ref{max-principle}.  $v(0) \le 0$ is clear since $u_K(t)$ and
  $P^U_t f$ are both continuous in $L^2$ as $t \downarrow 0$, and $f
  \le u(0)$.  Since $0 \le v^+(t) \le P^U_t f$ for almost every $t$,
  and $P_t^U f \in \mathbb{D}(U)$, we
  have $v^+(t) \in \mathbb{D}(U)$ for almost every $t \in [0,T']$ by Lemma
  \ref{control-bdy-cond}.  Finally, suppose $\phi \in L^2([0,T];
  \mathbb{D}(U))$ with $\phi(t) = 0$ outside some $(\delta,
  T'-\delta)$.  
  Since $\phi(t)$ is supported inside $\closure{U}$ for almost every
  $t$, $\phi$ is compactly supported inside $(0,T') \times X$.  Thus
  (\ref{u-max-weak}), with $T$ replaced by $T'$, holds for $u_K$
  (since $u$ is a local weak solution) and also for $P^U_t f$ (by
  spectral theory), so it holds for $v$.
  Thus Theorem \ref{max-principle} applies and we have $v \le
  0$ on $[0,T'] \times U$, which is to say $u(t) \ge P^U_t f$.
  Letting $T'  \to T$ we have the same on $[0,T) \times U$, and
  outside of $U$ we have $u(t) \ge 0 = P_t^U f$, so in fact $u(t) \ge
  P^U_t f$ on $[0,T) \times X$.

  Now take an increasing sequence of open precompact subsets $U_i
  \uparrow X$ (this is possible in any separable locally compact
  metric space).  We have $P_t^{U_i} f \le u$ for each $i$, and
  $P_t^{U_i} f \to P_t f$ a.e. by Theorem \ref{exhaustion-semigroup},
  so the proof is complete.
\end{proof}

In the following lemma, we show that nonnegative local weak solutions
have bounded $L^1$ norms on compact sets $K$ near the initial time.
This says, in some sense, that heat cannot flow out of $K$ too
rapidly.  We will use this fact in conjunction with weak compactness
to produce the measure $\nu$ in the main theorem.

\begin{proposition}\label{u-local-L1-bounded}
  Let $u$ be a nonnegative local weak solution on $(0,T) \times X$,
  and let $K \subset X$ be compact.  Then for any $T' < T$, we have
  \begin{equation*}
    \sup_{t \in (0, T')} \int_{K} u(t,y)\,\mu(dy) < \infty.
  \end{equation*}
\end{proposition}

\begin{proof}

  Fix any $0 < T' < T'' < T$.  For any $t \in (0,T')$, let $v(s,x) =
  u(s+t, x)$, so that $v$ is a nonnegative local weak solution on
  $(-t, T-t) \times X$, and $v(0,x) = u(t,x) \ge 1_K(x) u(t,x)$.
  Applying Proposition \ref{minimal} to $v$ with $f(x) = 1_K(x)
  u(t,x)$ gives
  \begin{equation*}
    P_{s}[1_K u(t,\cdot)](x) \le v(s, x) = u(s+t, x), \quad
    \text{$\mu$-a.e. $x$}
  \end{equation*}
  for any $0 \le s < T-t$.  Taking $s = T'' - t$, this reads
  \begin{equation}\label{P-order}
    P_{T'' - t}[1_K u(t,\cdot)](x) \le u(T'', x), \quad \text{$\mu$-a.e. $x$}.
  \end{equation}
  Let $A \subset X$ be any compact set of positive measure, and set
  \begin{equation*}
    c = \inf\{p(s,x,y) : s \in [T'' - T', T''], x \in A, y \in K\}.
  \end{equation*}
  Since we have assumed that the heat kernel $p$ is positive and
  continuous on $(0, \infty) \times X \times X$, we have $c > 0$.
  Then for $\mu$-a.e. $x \in A$ we have
  \begin{align*}
    P_{T'' - t}[1_K u(t, \cdot)](x) &= \int_K u(t,y) p(T'' - t, x,
    y)\,\mu(dy) \\
    &\ge c \int_K u(t,y)\,\mu(dy)
  \end{align*}
  and so, combining this with (\ref{P-order})
  \begin{equation*}
    \int_{K} u(t, y)
    \mu(dy) \le \frac{1}{c} \essinf_{x \in A} u(T'',
    x).
  \end{equation*}
  The right side is finite and
  independent of $t \in (0,T')$, so the proof is complete.
\end{proof}

\section{Widder's theorem}\label{widder-sec}

In this section we prove our main result.

\begin{theorem}[Widder's theorem, local version]\label{widder-local}
  Let $U \subset X$ be open, connected, nice and precompact, and
  suppose $u$ is a nonnegative local weak solution on $(0,T) \times W$ for some
  open neighborhood $W$ of $\closure{U}$.  Then, there exists a unique
  positive Radon measure $\nu$ supported on $U$, and a unique
  nonnegative local weak solution $h$ on $(-\infty, T) \times U$ with
  $h(t,\cdot) = 0$ for $t \le 0$, such that
  \begin{equation}
    u(t,x) = P_t \nu(x) + h(t,x), \quad t \in (0,T).
  \end{equation}
\end{theorem}

\begin{proof}
  For $\epsilon > 0$, set
  \begin{equation*}
    h_\epsilon(t) = 
    \begin{cases}
      u(t) - P_{t-\epsilon} [ 1_U u(\epsilon)], & \epsilon < t < T \\
      0, & -\infty < t \le \epsilon.
    \end{cases}
  \end{equation*}
  Since $1_U u(\epsilon) \in L^2(X,\mu)$, $h_\epsilon$ is a local weak
  solution on $(\epsilon, T) \times U$ (Lemma \ref{Pt-lws}).  By
  Proposition \ref{minimal} (shifting time by $\epsilon$), we also
  have $h_\epsilon \ge 0$ on $(-\infty, T) \times U$.  

  For any nice cutoff function $\psi$ supported inside $U$,
  we have $\psi u \in L^2([\epsilon, T-\epsilon]; \mathbb{D})$ by
  definition of local weak solution.  We also have $P_t [1_U
  u(\epsilon)] \in L^2([\epsilon, T-\epsilon]; \mathbb{D})$ by Lemma
  \ref{Pt-lws}, so the same holds for $\psi P_t[1_U u(\epsilon)]$.
  And by strong continuity of the heat semigroup $P_t$, we have $\psi
  h_{\epsilon}(t) \to 0$ in $L^2(X,\mu)$ as $t \to \epsilon$.  So by Lemma
  \ref{extension}, $h_\epsilon$ is a local weak solution on $(-\infty, T)
  \times U$.

  Now by Proposition \ref{u-local-L1-bounded} with $K = \closure{U}$, we have that
  $1_U u(\epsilon)$ is bounded in $L^1$ norm as $\epsilon \downarrow 0$, or
  equivalently, that the Radon measures $1_U u(\epsilon) d\mu$ are bounded
  in total variation.  Hence by compactness, there is a sequence
  $\epsilon_n \downarrow 0$ and a positive Radon measure $\nu$ with
  $1_U u(\epsilon_n) d\mu \to d\nu$ weakly.  By Lemma
  \ref{Pt-weak-cts}, we have $P_{t - \epsilon_n} [1_U u(\epsilon_n)](x)
  \to P_t \nu(x)$ pointwise.  Thus $h_{\epsilon_n}(t,x) \to h(t,x) =
  u(t,x) - P_t \nu(x)$ pointwise, where we take $h(t,x) = 0$ for $t
  \le 0$.

  We now apply the parabolic Harnack inequality (Assumption
  \ref{harnack}).  Fix $-\infty < a < 0 < b < c < d < T$, and $V$ open and
  precompact with $\closure{V} \subset U$.  By the Harnack
  inequality, for each $h_{\epsilon}$ we have
  \begin{equation*}
    \esssup_{[a,b] \times \closure{V}} h_{\epsilon} \le C
    \essinf_{[c,d] \times \closure{V}} h_\epsilon \le C \essinf_{[c,d]
    \times \closure{V}} u
  \end{equation*}
  since $h_\epsilon \le u$.  Since the bound is independent of
  $\epsilon$, we can apply Lemma \ref{weak-soln-limit} to find that $h
  = \lim_{n \to \infty} h_{\epsilon_n}$ is a local weak solution on
  $(a,b) \times V$, and hence (since $a,b,V$ were arbitrary) on
  $(-\infty, T) \times U$.  This completes the proof of existence.

  To show uniqueness of $\nu$ and $h$, fix $f \in C_c(U)$.  Lemma
  \ref{Pt-weak-lim-zero} says that we have $\int_X  P_t \nu(x) f(x)\,
  \mu(dx) \to \int_X f\,d\nu$ as $t \to 0$.  Since $h$ is a local weak
  solution on $(-\infty, T) \times U$, $h$ is continuous in
  $L^2_\loc(U)$, and since $h$ vanishes for $t \le 0$, we have $\int_X
  h(t,x) f(x)\,\mu(dx) \to 0$.  Thus $\int_X u(t,x) f(x) \,\mu(dx) \to
  \int f \,d\nu$, which shows that $\nu$, and therefore also $h$, is
  uniquely determined by $u$.  In fact, we have shown that $u(t) \to
  \nu$ weakly on $U$ (i.e. in the weak-* topology of $C_c(U)^*$).

\end{proof}

\begin{theorem}[Widder's theorem, global version]\label{widder-global}
   Let $u$ be a nonnegative local weak solution on $(0,T) \times X$.
  There exists a unique positive Radon measure $\nu$ (possibly
  infinite), and a unique nonnegative local weak solution $h$ on
  $(-\infty, T) \times X$ with $h(t,x) = 0$ for $t \le 0$, such that
  \begin{equation}\label{widder-global-eq}
    u(t,x) = P_t \nu(x) + h(t,x), \quad t \in (0,T).
  \end{equation}
\end{theorem}

\begin{proof}
  Let $U_n$ be an increasing exhaustion of $X$ by open, precompact,
  connected, nice sets, and for each $U_n$ let $u(t) = P_t \nu_n +
  h_n(t)$ be the unique decomposition produced by Theorem
  \ref{widder-local}.  As we previously argued, $u(t) \to \nu_n$
  weakly on $U_n$ as $t \downarrow 0$, and thus for $m > n$, we have
  $\nu_m = \nu_n$ on $U_n$.  In particular, the measures $\nu_n$ are
  increasing.  Their limit $\nu$ is another positive Radon measure,
  possibly infinite, and by monotone convergence we have $P_t \nu_n
  \uparrow P_t \nu$.  Thus $h_n \downarrow h = u - P_t \nu$.  $h$
  remains a nonnegative function which vanishes for $t \le 0$.

  Moreover, if $V$ is any precompact open set, we have $\closure{V}
  \subset U_n$ for sufficiently large $n$.  Fixing $-\infty < a < 0 <
  b < c < d$, we have by the Harnack inequality
  \begin{equation*}
    \esssup_{[a,b] \times \closure{V}} h_n \le C
    \essinf_{[c,d] \times \closure{V}} h_n \le C \essinf_{[c,d]
    \times \closure{V}} u
  \end{equation*}
  Thus applying Lemma \ref{weak-soln-limit}, $h$ is a local weak
  solution on $(a,b) \times V$, and hence on $(-\infty, T) \times X$.

  For uniqueness, suppose we have another decomposition $u(t) = P_t
  \tilde{\nu} + \tilde{h}(t)$.  Let $f \in C_c(X)$ be nonnegative; $f$
  is supported in one of the $U_n$, so as $t \to 0$ we have
  \begin{equation*}
    \int_X u(t,x) f(x)\, \mu(dx) \to \int_{U_n} f(x) \,d\nu_n = \int_X f\,d\nu.
  \end{equation*}
  Thus since $\tilde{h}(t) \to 0$ in $L^2_\loc(X)$, we also have 
  \begin{equation*}
    \int_X P_t f \,d \tilde{\nu} = \int_X P_t \tilde{\nu}(x)
    f(x)\,\mu(dx)  \to \int_X f \,d\nu.
\end{equation*}
  Since $P_t f \to f$ pointwise, Fatou's lemma gives $\int_X
  f\,d\tilde{\nu} \le \int_X f\,d\nu$.  Thus we have $\tilde{\nu} \le
  \nu$, so $\eta = \nu - \tilde{\nu}$ is a positive Radon measure.
  Since $P_t \eta = \tilde{h}(t) - h(t)$, we have that $P_t \eta \to
  0$ in $L^2_\loc(X,\mu)$ as $t \to 0$.  If $K$ is any compact set
  and $\eta|_K$ is the restriction of $\eta$ to $K$, then $0 \le
  P_t \eta|_K \le P_t \eta$, so we also have $P_t \eta|_K \to 0$ in
  $L^2_\loc$.  However, since $\eta|_K$ is a compactly supported Radon
  measure, Lemma \ref{Pt-weak-lim-zero} gives $P_t \eta|_K \to
  \eta|_K$ weakly, so $\eta|_K = 0$.  Letting $K \uparrow X$, we have
  $\eta = 0$, and thus $\nu = \tilde{\nu}$ so the decomposition is
  unique.
\end{proof}

\section{Conditions for uniqueness of nonnegative solutions}\label{unique-sec}

In the classical version of Widder's theorem (for the classical heat
equation on $\R^d$), the function $h$ appearing in
(\ref{widder-global-eq}) is actually zero, and the theorem just states
that $u(t,x) = P_t \nu(x)$.  Thus in the classical case, a nonnegative
solution of the heat equation is uniquely determined by its initial
values.  However, in our general setting, $h$ can certainly
be nonzero.  For example, let $X = (0,\infty)$ be the open half-line,
with the classical Dirichlet form $\mathcal{E}(f,g) = \frac{1}{2} \int_0^\infty f'
g' \,dm$ with its domain $\mathbb{D} = H^1_0((0, \infty))$.  Then
$u(t,x) = \frac{1}{\sqrt{2 \pi t}} e^{-x^2/2t}$ is a local weak
solution, but it is easy to see that the decomposition in
(\ref{widder-global-eq}) must have $\nu = 0$ and $h = u$.  

We record here some conditions that are, or are not, necessary or
sufficient to guarantee that $h=0$.

\begin{enumerate}
  \item The Dirichlet space $(X,\mu, \mathcal{E},\mathbb{D})$ is said
    to be \textbf{conservative} (or \textbf{stochastically complete})
    if $P_t 1 = 1$, or equivalently if the corresponding Hunt process
    $X_t$ has an infinite lifetime, almost surely.  This condition is
    necessary, but not sufficient, to ensure $h=0$.

    To see it is necessary, observe that $v(t,x) = 1 - (P_t 1)(x)$
    satisfies the hypotheses of Lemma \ref{extension}, and hence can
    be regarded as a nonnegative local weak solution on $(-\infty,
    \infty) \times X$.  Applying Theorem \ref{widder-global} to
    $u(t,x) = v(t-1, x)$, we have $v(t-1, x) = P_t \nu(x) + h(t,x)$.
    If $h=0$ then $v(t-1, x) = P_t \nu(x)$, but since this vanishes
    for all $0 \le t \le 1$, we must have $\nu = 0$ and hence $v=0$
    identically.

    To see it is not sufficient, see the next example.

  \item A stronger condition is that the Dirichlet space, or
    equivalently its corresponding Hunt process, be
    \textbf{recurrent}.  This is also not sufficient to ensure $h=0$.
    Consider $X = \R^2 \setminus \{0\}$ with the classical Dirichlet
    form $\mathcal{E}(f,g) = \frac{1}{2} \int \grad f \cdot \grad g \,dm$ and its
    domain $\mathbb{D} = H^1_0(X)$.  Since points are polar for
    Brownian motion in $\R^2$,
    this is a recurrent Dirichlet space.  However, it is not hard
    to see that
    \begin{equation*}
      u(t,x) = \begin{cases} \frac{1}{2 \pi (t-1)} e^{-|x|^2/2(t-1)}, & t > 1
      \\
      0, & t \le 1
      \end{cases}
    \end{equation*}
    is a nonnegative local weak solution on $(0, \infty) \times X$.
    Since it vanishes for $0 \le t \le 1$, its decomposition according
    to Theorem \ref{widder-global} must have $\nu = 0$, so $h$ cannot
    be $0$.
    
    Recurrence is not necessary, as can be seen by considering the
    classical Dirichlet form on $\R^d$, $d \ge 3$.  The fact that $h=0$
    in this case is Widder's original theorem; it is also included in
    the Harnack-type case below.
        

  \item Completeness under an intrinsic metric is not sufficient to
    ensure $h=0$.  See \cite[Section 7.7]{azencott74} for an example
    of a complete two-dimensional Riemannian manifold with unbounded
    negative curvature, such that the Brownian motion explodes in
    finite time with positive probability.  This corresponds to a
    strictly local Dirichlet space which is not conservative.
    
  \item If $X$ is compact, then any local weak solution $u$ on $I
    \times X$ is
    actually a global weak solution, because we can take $K = [a,b]
    \times X$ in Definition \ref{lws-def}.  In particular, we have
    $u(t) \in \mathbb{D}$ for all $t$.  So we can apply the maximum
    principle (Theorem \ref{max-principle}) to $h$ and immediately
    conclude that $h=0$.  In fact, when $X$ is compact, many of the
    arguments in this paper become much simpler.

\item Under the basic assumptions made in Section
\ref{assumptions-sec}, if we assume further that either (ID2) holds
locally uniformly or that (H-$\beta$) holds locally uniformly then we
can follow the elegant argument of \cite{koranyi-taylor}.  A
nonnegative local weak solution $u$ is said to be \textbf{minimal} if
the only local weak solutions $v$ satisfying $0 \le v \le u$ are of
the form $v = \lambda u$; the Choquet representation theorem says that
any solution has an integral representation in terms of minimal
solutions.  Suppose, then, that $h$ is a nonnegative local weak
solution which vanishes for $t \le 0$; it has a representation in
terms of minimal solutions $\tilde{h}$ that also vanish for $t \le 0$.
However, it follows from the locally uniform parabolic Harnack
inequality that for sufficiently small $\epsilon$, we have
$\tilde{h}(t-\epsilon, x) \le H_0 \tilde{h}(t,x)$; thus by minimality
$\tilde{h}(t-\epsilon, x) = \lambda \tilde{h}(t,x)$ and we conclude
that $\tilde{h}$ vanishes for $t \le \epsilon$.  By iteration,
$\tilde{h}$ vanishes everywhere, and so the same must be true of $h$.

In fact, \cite{koranyi-taylor} proves the much stronger statement that 
nonnegative minimal weak solutions $u$ of (\ref{heat-eqn}) on
$(-\infty,T)\times X$ are in fact of the form
$$u(t,x)=e^{\alpha t}v(x)$$
where $v$ is a nonnegative minimal weak solution of $Lv=\alpha v$ on $X$.

\end{enumerate}

The question of whether a nonnegative solution of the heat equation is
uniquely determined by its initial values has been studied by many
authors in various settings.  In addition to \cite{widder44} and
\cite{aronson68,aronson68add}, we mention 
\cite{ishige-murata, murata-2003, murata-2005,koranyi-taylor}.


\section{An application to projections} In this short section we outline 
what we think is a compelling application of our main result to the study 
of the projection of one Dirichlet space onto another. 

Let $(X_i,d_i,\mu_i,\mathcal E_i,\mathbb D_i)$, $i=1,2$ 
be two Dirichlet spaces satisfying the assumptions of Section 
\ref{assumptions-sec}. Assume further that $(X_i,d_i)$, $i=1,2$ are complete.
We are interested in considering the situation when there exists 
a continuous projection map $\pi: X_1 \rightarrow X_2$ 
with the following properties:
\begin{itemize}
\item[(P1)] For any $x,y,\tilde{x}$ with $x,y\in X_2$, $\tilde{x}\in X_1$ with
 $\pi(\tilde{x})=x$,
$$d_2(x,y)=\min\{d_1(\tilde{x},\tilde{y}): \tilde{y}\in \pi^{-1}(\{y\})\}.$$
\item[(P2)] If $u$ is a local weak solution of (\ref{heat-eqn}) in 
$(0,T)\times U$ on $X_2$ then $v(t,x)=u(t,\pi(x))$ is a local weak solution of
(\ref{heat-eqn}) in $(0,T)\times \tilde{U}$ on $X_1$ where 
$\tilde{U}=\pi^{-1}(U)$.
\item[(P3)] A Borel set $A\subset X_2$ is $\mu_2$-negligible if and only if 
$\pi^{-1}(A)$ is $\mu_1$-negligible.  
\end{itemize}
If $B_i(x,r)$ denotes the ball of radius $r$ around $x$ in $X_i$, $i=1,2$, 
then (P1) implies that $$\pi^{-1}(B_2(x,r))\supset B_1(\tilde{x},r)$$ 
and that 
$$\pi(B_1(\tilde{x},r))\subset B_2(x,r)$$
for any $\tilde{x}$ such that $\pi(\tilde{x})=x$. 

Condition (P3) implies that we can disintegrate $\mu_1$ with respect to $\mu_2$; 
that is, there exists a family of measures $\nu^\pi_z$ on $X_1$, $z\in X_2$,  
with $\nu^\pi_z$ supported in $\pi^{-1}(z)$ and such that
\begin{equation}
\int_{X_1}f(x)\,d\mu_1(x)=\int_{X_2}\int_{X_1}f(x)\,d\nu^\pi_z(x)\,d\mu_2(z)
\end{equation}
for any nonnegative measurable $f$ on $X_1$. We will assume that 
this disintegration formula has the following continuity property. 
\begin{itemize}
\item[(P4)] For $f\in C_c(X_1)$, the measurable 
compactly supported function
$$z\mapsto \int f\,d\nu^\pi_z$$
admits a continuous version.
\end{itemize}

\begin{theorem} \label{th-proj}
Referring to the setup introduced above, assume 
that {\em (P1), (P2), (P3)} and {\em (P4)} hold true and that 
$(X_1,d_1,\mu_1,\mathcal E_1,\mathbb D_1)$ 
satisfies the parabolic 
Harnack inequality {\em (H-$\beta$)} of {\em Section \ref{example-sec}}, 
locally uniformly. Then the same is true for 
$(X_2,d_2,\mu_2,\mathcal E_2,\mathbb D_2)$. Furthermore, the 
two heat kernels are related by
$$p_2(t,x,y)= \int p_1(t,\tilde{x},z)\,d\nu^\pi_{y}(z)$$
where $\tilde{x}$ is such that $\pi(\tilde{x})=x$. 
\end{theorem}
\begin{proof}
The first assertion follows immediately by inspection, using (P1) and (P2)
to lift a local solution on $X_2$ to a local solution on $X_1$.

Consider 
$$(t,\tilde{z})\mapsto u_y(t,\tilde{z})= p_2(t,\pi(\tilde{z}),y).$$ 
This is a nonnegative  
local weak solution of (\ref{heat-eqn}) on $(0,\infty)\times X_1$.
By Theorem \ref{widder-global} (with $h\equiv 0$ because of the validity 
of the local uniform parabolic Harnack inequality) there exists a nonnegative 
Radon measure $\omega_y$  on $X_1$ such that, for all $\tilde{z}\in X_1$ 
and $t>0$,
$$u_y(t,\tilde{z})=\int_{X_1}p_1(t,\tilde{z},\zeta)\,d\omega_y(\zeta)$$
and, for any $f\in C_c(X_1)$, as $t$ tends to $0$,
$$\int_{X_1}u_y(t,\zeta)f(\zeta)\,d\mu_1(\zeta)
=\int_{X_1}p_2(t,\pi(\zeta),y)f(\zeta)\,d\mu_1(\zeta)
\rightarrow \int_{X_1}f(\zeta)\,d\omega_y(\zeta).$$

Now,
$$\int_{X_1}p_2(t,\pi(\zeta),y)f(\zeta)\,d\mu_1(\zeta)
= \int_{X_2}\left(\int_{X_1}f\,d\nu^\pi_{z}\right)p_2(t,z,y)\,d\mu_2(z).$$
Since $z\mapsto \int_{X_1}fd\nu^\pi_z$ admits a continuous version, we 
see that
$$\int_{X_1}p_2(t,\pi(\zeta),y)f(\zeta)\,d\mu_1(\zeta))
\rightarrow \int_{X_1}f\,d\nu^\pi_{y}.$$
In other words, the Radon measure $\omega_y$ is, in fact, equal to $\nu^\pi_y$.
\end{proof}
 
Theorem \ref{th-proj} is surprising and interesting even in the simplest cases.
Consider for instance the case when $X_1$ and $X_2$ are complete Riemannian 
manifolds, each equipped with its natural Dirichlet space structure, and $\pi$ 
is the projection associated with a countable group $G$ of isometries
acting properly and freely on $X_1$.
The hypotheses (P1)--(P4) are clearly satisfied (the measure $\nu^\pi_z$ is 
the counting measure on the countable set $\pi^{-1}(z)$.) The theorem
says that the heat kernels on $X_1$ and $X_2$ are related by
$$p_2(t,\pi(x),\pi(y))=\sum_{g\in G}p_1(t,x,gy).$$
This statement includes the non-trivial fact that the sum 
on the right hand side is finite.  

Another illustrative  application of Theorem \ref{th-proj} is 
to relate the Gaussian heat kernel of Brownian motion on $\mathbb R^n$ 
to its radial part, the Bessel process, which is associated with an explicit
Dirichlet space on the semi-axis (this requires a proper treatment of the point 
$0$, depending on dimension). See  \cite[page 126]{chen-fukushima-book}. In this case, 
the group action is the action of the rotation group.

The setting of Dirichlet spaces allow us to treat in exactly the same
way the very natural case that arises when $X_1$ and $X_2$ are
polytopal complexes (satisfying mild assumptions, see
\cite{pivarski-saloff-coste-complex, eells-fuglede}) instead of
Riemannian manifolds. See also \cite{bendikov-saloff-coste-diagonal,
bendikov-saloff-coste-invariant} for examples involving the Dirichlet
forms of Section \ref{sec-torus} on $\mathbb R^k\times\mathbb
T^\infty$.

Theorem \ref{th-proj} can be applied in a wide variety of contexts where 
the projection $\pi$ is associated to the proper continuous action of a 
locally compact group $G$ on $X_1$ that preserves the distance $d_1$.
See for instance \cite{bssw-strip-complex}, especially Corollary 4.6
and Section 6, for descriptions of
concrete examples.

We end this section by specializing Theorem \ref{th-proj} in the 
context of sub-Riemannian diffusions on unimodular groups.

\begin{theorem} Let $G$ be a unimodular Lie group equipped with its 
Haar measure and a family $\{X_1,\dots,X_k\}$ of left invariant vector fields
generating the Lie algebra of $G$. Let $H$ be a closed subgroup of $G$ 
equipped with its Haar measure. Let
$\pi: G\rightarrow M=H\setminus G$ be the projection on the quotient space $M$
of right-cosets.
Equip $M$  with its natural $G$-invariant measure. 
Let $ L_G= \sum_1^k X_i^2$ and $L_M=\sum_1^k [d\pi(X_i)]^2$ be the 
associated hypoelliptic sub-Laplacians on $G$ and $M$, respectively. 
Then the heat kernels on $M$ and $G$ are related by
$$p_M(t,x,y)=\int_H p_G(g_x,hg_y) d_Hh$$
where $\pi(g_x)=x$, $\pi(g_y)=y$ and $d_Hh$ is the Haar measure on $H$. 
In particular, the right-hand side is finite.
\end{theorem}

For background information regarding the setting of this theorem, see
\cite{purplebook, maheux-heat-kernel}.

\section{Acknowledgements}

The authors wish to thank Benjamin Steinhurst for many helpful
discussions.  We would also like to thank the anonymous referee for a
very careful reading of the paper which led to numerous corrections
and improvements.

N.~Eldredge was partially supported by National Science
Foundation grant DMS-0739164.  L.~Saloff-Coste was partially supported
by National Science Foundation grant DMS-1004771.

\appendix

\section{Energy measures}\label{app-energy-measure}

Let $(\mathcal{E}, \mathbb{D})$ be a regular, strictly local Dirichlet
form on $L^2(X,\mu)$.

Recall that $\mathbb{D} \cap L^\infty(X,\mu)$ is an algebra, and for
$f,g \in \mathbb{D} \cap L^\infty$ we have
\begin{equation*}
  \sqrt{\mathcal{E}(fg)} \le \norm{f}_\infty
  \sqrt{\mathcal{E}(g)} + \norm{g}_\infty
  \sqrt{\mathcal{E}(f)}.
\end{equation*}
(See \cite[Theorem 1.4.2 (ii)]{fukushima-et-al}.)

For $f \in \mathbb{D} \cap L^\infty$, we can define a Radon measure
$\Gamma(f)$ on $X$ by taking
\begin{equation}\label{gamma-def}
  \int_X \phi d\Gamma(f) := 2 \mathcal{E}(\phi f, f) -
  \mathcal{E}(f^2, \phi)
\end{equation}
for $\phi \in C_c(X) \cap \mathbb{D}$.  (In the classical case,
$d\Gamma(f) = |\grad f|^2 dm$, where $m$ is Lebesgue measure.)

To see that (\ref{gamma-def}) in fact defines a Radon measure, i.e. a
continuous linear functional on $C_c(X)$, set
\begin{equation*}
  \alpha_t(\phi) := \frac{1}{t}[2 \inner{f \phi}{P_t f - f} -
    \inner{P_t(f^2) - f^2}{\phi}]
\end{equation*}
and note that
\begin{equation*}
  |\alpha_t(\phi)| \le \frac{1}{t} [2 |\inner{f}{P_t f - f}| + ||P_t
    (f^2) - f^2||_{L^1}] ||\phi||_\infty
\end{equation*}
where $||P_t (f^2) - f^2||_{L^1} < \infty$ because $f \in L^2$, so
$f^2 \in L^1$, and $P_t$ is a contraction on $L^1$ (which follows from
the Markovian property).  So $\alpha_t$ is a bounded linear functional
on $C_c(X)$, and as $t \to 0$, $\alpha_t(\phi) \to 2 \mathcal{E}(\phi
f, f) - \mathcal{E}(f^2, \phi)$.  By the uniform boundedness
principle, a pointwise limit of bounded linear functionals is another
bounded linear functional.

One may then define the signed measure $\Gamma(f,g) =
\frac{1}{2}(\Gamma(f+g) - \Gamma(f) - \Gamma(g))$ by polarization,
where $\Gamma(f,f) = \Gamma(f)$.

Note that for $f,g \in \mathbb{D} \cap L^\infty$, the integral $\int_X
f \,d\Gamma(g)$ needs some care to be well-defined, since $f$ is
technically only defined up to $\mu$-null sets, which $\Gamma(g)$ may
charge.  However, a quasi-continuous $\mu$-version $\tilde{f}$ of $f$
is uniquely defined up to polar sets, which $\Gamma(g)$ does not
charge.  So $\int_X f\,d\Gamma(g)$ should be interpreted as $\int_X
\tilde{f}\,d\Gamma(g)$.

Note that \cite{fukushima-classic} discusses energy measures for
additive functionals, but this is a generalization of the energy
measure of a function.  Note also in the strictly local case we have,
in the notation of \cite{fukushima-classic}, 
$\mathcal{E}^{\mathrm{res}} = \mathcal{E}$.

Some properties of $\Gamma$ which we shall use:

\begin{proposition}[{\cite[Lemma 5.4.2]{fukushima-classic}}]\label{energy-leibniz}
  For $f,g,h \in \mathbb{D} \cap L^\infty$, $d\Gamma(fg,h) = f
  d\Gamma(g,h) + g d\Gamma(f,h)$.
\end{proposition}

\begin{proposition}[{\cite[Lemma 5.4.3]{fukushima-classic}}]\label{energy-cauchy-schwarz}
For $f,g,h,k \in \mathbb{D} \cap L^\infty$, we have the
Cauchy--Schwarz inequality
\begin{equation*}
  \left(\int_X |fg|\, d|\Gamma(h,k)|\right)^2 \le \int_X f^2 \,d\Gamma(h)
  \int_X g^2 \,d\Gamma(k).
\end{equation*}
\end{proposition}

We remark that using the AM-GM inequality, we have the useful form
\begin{equation}\label{energy-amgm}
  \abs{\int_X fg \,d\Gamma(h,k)} \le \int_X |fg| \,d|\Gamma(h,k)| 
  \le \frac{1}{2} \left(\int_X f^2 \,d\Gamma(h)
  +\int_X g^2 \,d\Gamma(k)\right).
\end{equation}
The latter form will be more useful to us.

\begin{corollary}\label{energy-product-bound}
  For $f, g \in \mathbb{D} \cap L^\infty$, 
  \begin{equation*}
    d\Gamma(fg) \le 2 (f^2 d\Gamma(g) + g^2 d\Gamma(f)).
  \end{equation*}
\end{corollary}

\begin{proof}
  Fix $h \in \mathbb{D} \cap C_c(X)$ with $h \ge 0$. By repeated
  application of the Leibniz rule (Proposition \ref{energy-leibniz}),
  \begin{align*}
    \int_X h \,d\Gamma(fg) &= \int_X h f^2\, d\Gamma(g) + \int_X h g^2
    \,d\Gamma(f) + 2 \int_X hfg\,d\Gamma(f,g). 
  \end{align*}
  Writing $hfg$ as $(f \sqrt{h})(g \sqrt{h})$ and applying
  (\ref{energy-amgm}), we have what we want.
\end{proof}

\section{Cutoff Sobolev inequalities}\label{cutoff-sobolev}

In \cite{bbk06}, the notion of a (local) cutoff Sobolev inequality is
defined.  Here $X$ is assumed to be a strictly local Dirichlet space.
$X$ was also assumed to be a metric space; we write $d$ for the metric
and $B(x,r)$ for the open metric balls.  Note that $d$ is not assumed
to be an intrinsic distance in the sense of the previous section.

\begin{definition}[{\cite{bbk06}}]\label{cutoff-sobolev-def}
  $X$ is said to satisfy a \textbf{local cutoff Sobolev inequality}
  $\mathrm{CS}(\beta)_{\mathrm{loc}}$ if there exists $\theta \in
  (0,1]$ and constants $c_1, c_2$ such that for every $x_0 \in X$, $0 <
  R \le 1$, there exists a cutoff function $\psi$ with the properties:
  \begin{enumerate}
  \item $\psi \ge 1$ on $B(x_0, R/2)$;\label{cutoff-sobolev-1}
  \item $\psi = 0$ on $B(x_0, R)^c$;\label{cutoff-sobolev-0}
  \item $|\psi(x) - \psi(y)| \le c_1 (d(x,y)/R)^\theta$ for all $x,y
    \in X$; \label{cutoff-sobolev-holder}
  \item \label{cutoff-sobolev-energy} For any ball $B(x,s)$ with $0 \le s \le R$ and $f \in
  \mathbb{D}$,
  \begin{equation}\label{cutoff-sobolev-energy-eqn}
    \int_{B(x,s)} f^2\,d\Gamma(\psi) \le c_2 (s/R)^{2\theta}
    \left(\int_{B(x,2s)} d\Gamma(f) + s^\beta \int_{B(x,2s)}
    f^2\,d\mu \right). 
  \end{equation}
  \end{enumerate}
\end{definition}

Note that by replacing $\psi$ with $\bar{\psi} = \psi \wedge 1 \vee
0$, we can assume $0 \le \psi \le 1$.  (If $\psi$ satisfies the above
condition, so does $\bar{\psi}$: we have $|\bar{\psi}(x) -
\bar{\psi}(y)| \le |\psi(x) - \psi(y)|$, and the Markovian properties
of $(\mathcal{E},\mathbb{D})$ implies $d\Gamma(\bar{\psi})
\le d\Gamma(\psi)$.)

We remark that if $\psi$ satisfies (\ref{cutoff-sobolev-energy-eqn}),
then for any $f \in \mathbb{D}$, we have $f \psi \in \mathbb{D}$.  We
clearly have $f \psi \in L^2(X,\mu)$ since $\psi$ is bounded.  If we
take $s=R$ in (\ref{cutoff-sobolev-energy-eqn}), then the left side
becomes $\int_X f^2 d\Gamma(\psi)$ and the right side is
controlled by $\mathcal{E}_1(f)$.  If $f \in \mathbb{D} \cap
L^\infty$ we can apply Corollary \ref{energy-product-bound} and see
that $\mathcal{E}_1(f \psi)$ is controlled by $\mathcal{E}_1(f)$.
Since by the Markovian property, $\mathbb{D} \cap L^\infty$ is
$\mathcal{E}_1$-dense in $\mathbb{D}$, an approximation argument shows
$f \psi \in \mathbb{D}$ for any $f \in \mathbb{D}$.

 \begin{lemma}
   If $X$ satisfies $\mathrm{CS}(\beta)_{\mathrm{loc}}$, then all open
  sets in $X$ are nice, and $X$ satisfies Assumption \ref{union-of-nice}.
\end{lemma}

\begin{proof}
  Let $U \subset X$ be open, and $K \subset U$ compact.  By
  compactness of $K$ and local compactness of $X$, we can cover $K$ by
  a finite number of balls $B(x_i, R_i/2)$ such that $R_i \le 1$ and
  $\closure{B(x_i, R_i)} \subset U$.  Let $\psi_i : X \to [0,1]$ be the
  corresponding cutoff functions as provided by Definition
  \ref{cutoff-sobolev-def}.  If $f \in \mathbb{D}$, then we have
  argued that $f \psi_i \in \mathbb{D}$ for each $i$.  Now taking
  $\psi = \max \psi_i$ gives a nice
  cutoff function for $K$ inside $U$, since $f \psi = \max\{f^+ \psi_i\}
  - \max\{f^- \psi_i\} \in \mathbb{D}$ as well.
\end{proof}

\section{Maximum principle}\label{app-max-principle}

In this appendix we give a proof of Theorem \ref{max-principle},
adapted from \cite{grigoryan-hu-off-diagonal}, for local weak solutions.

\begin{lemma}\label{chain-rule}
  Let $u \in C^1([0,T]; L^2(X,\mu))$, and suppose $\Phi \in C^1(\R)$
  satisfies $\Phi(0) = 0$ and $|\Phi'| \le 1$.  Then $\Phi \circ u \in
  C([0,T]; L^2(X,\mu))$, and $\Phi \circ u$ is differentiable with
  $(\Phi \circ u)'(t,x) = \Phi'(u(t,x))u'(t,x)$.
\end{lemma}

\begin{proof}
  Note that $\Phi$ is Lipschitz, so if $f \in L^2$, then $|\Phi
  \circ f| \le |f|$ and so $\Phi \circ f \in L^2$.  Also, if $f,g \in
  L^2$, then $|\Phi \circ f - \Phi \circ g| \le |f-g|$ and so $f
  \mapsto \Phi \circ f$ is a continuous function on $L^2$ (indeed, Lipschitz).  Thus $\Phi
  \circ u \in C([0,T]; L^2)$.
  
  For the derivative, we must show that for each $t \in (0,T)$ and every real sequence
  $\epsilon_n \to 0$,
  \begin{equation}\label{deriv-toshow}
    \lim_{\epsilon_n \to 0} \frac{\Phi(u(t+\epsilon_n, \cdot)) -
    \Phi(u(t,\cdot))}{\epsilon_n} = \Phi'(u(t,\cdot))u'(t,\cdot)
  \end{equation}
   with the convergence in $L^2$.  
   Let us write
   \begin{align*}
     \frac{\Phi(u(t+\epsilon_n, x)) -
    \Phi(u(t,x))}{\epsilon_n} &= \frac{\Phi(u(t,x) + \epsilon_n
    u'(t,x)) - \Phi(u(t,x))}{\epsilon_n} \\ 
     &\quad + \frac{\Phi(u(t+\epsilon_n,
    x)) - \Phi(u(t,x) + \epsilon_n u'(t,x))}{\epsilon_n}.
   \end{align*}
   The first term converges to $\Phi'(u(t,x)) u'(t,x)$ pointwise, and
   since $\Phi$ is Lipschitz we also have 
   \begin{equation*}
     \abs{\frac{\Phi(u(t,x) + \epsilon_n
    u'(t,x)) - \Phi(u(t,x))}{\epsilon_n}} \le |u'(t,x)| \in L^2
   \end{equation*}
   so by dominated convergence, this convergence is also in $L^2$.
   For the second term, we have, again because $\Phi$ is Lipschitz, that
   \begin{equation*}
     \abs{\frac{\Phi(u(t+\epsilon_n,
       x)) - \Phi(u(t,x) + \epsilon_n u'(t,x))}{\epsilon_n}} \le
       \abs{\frac{u(t+\epsilon_n, x) - u(t,x) -
       \epsilon_n u'(t,x)}{\epsilon_n}}
   \end{equation*}
   which goes to $0$ in $L^2$ by differentiability of $u$.
\end{proof}

\begin{notation}
  As in \cite[Proposition 4.11]{grigoryan-hu-off-diagonal}, set
  \begin{align*}
    \varphi(s) = \begin{cases} e^{-s^{-2}}, & s > 0 \\
      0, & s \le 0
      \end{cases}
  \end{align*}
  and $\Phi(s) = \left(\int_{-\infty}^s
  \varphi(\xi)\,d\xi\right)^{1/2}$, so that $\varphi = 2 \Phi \Phi'$.  Then one can
  verify that:
  \begin{itemize}
    \item $\varphi, \Phi \in C^1(\R)$;
    \item $\varphi > 0$ and $\Phi > 0$ on $(0, \infty)$;
    \item $0 \le \varphi' \le 1$ and $0 \le \Phi' \le 1$ on $\R$.
  \end{itemize}
  We remark in particular that for $f \in \mathbb{D}$, $\varphi \circ
  f$ is a normal contraction of $f$, and thus $\mathcal{E}_1(\varphi
  \circ f) \le \mathcal{E}_1(f)$.  Also, it is shown in \cite[Lemma
  4.3]{grigoryan-hu-off-diagonal} that 
  \begin{equation}\label{gh43}
    \mathcal{E}(f, \varphi \circ
  f) \ge \mathcal{E}(\varphi \circ f) \ge 0.
  \end{equation}
\end{notation}

\begin{lemma}\label{Phi-lemma}
For any $v \in W^{1,2}([0,T]; \mathbb{D}, \mathbb{D}^*)$, the function
$t \mapsto \norm{\Phi(v(t, \cdot))}_{L^2}^2$ is absolutely continuous,
and 
\begin{equation*}
  \frac{d}{dt} \norm{\Phi(v(t, \cdot))}_{L^2}^2 
  = (v'(t), \varphi(v(t, \cdot)))_{\mathbb{D}^*, \mathbb{D}}.
\end{equation*}
\end{lemma}

\begin{proof}
  Suppose first that $v \in C^1([0,T]; \mathbb{D})$.  We have by
  the product rule
 and Lemma \ref{chain-rule} that
  \begin{align*}
    \frac{d}{dt} \norm{\Phi(v(t, \cdot))}_{L^2}^2 &= 2 \int_X
    \Phi'(v(t,x)) v'(t,x) \Phi(v(t,x))\,\mu(dx) \\ 
    &= \int_X v'(t,x) \varphi(v(t,x))\,\mu(dx) = (v'(t), \varphi(v(t,
    \cdot)))_{\mathbb{D}^*, \mathbb{D}}.
  \end{align*}
  Thus the lemma holds for such $v$.
  Integrating by parts gives, for any $\chi \in C^\infty_c((0,T))$,
  \begin{equation}\label{Phi-eq}
    \int_0^T \norm{\Phi(v(t, \cdot))}_{L^2}^2 \chi'(t)\,dt = -
    \int_0^T \chi(t) (v'(t), \varphi(v(t, \cdot)))_{\mathbb{D}^*,
    \mathbb{D}}\,dt.
  \end{equation}

    Now suppose $v \in W^{1,2}([0,T]; \mathbb{D}, \mathbb{D}^*)$.  
    Since $C^1([0,T]; \mathbb{D})$ is dense in $W^{1,2}$, we can find a sequence $v_n \in
    C^1([0,T]; \mathbb{D})$ with $v_n \to v$ in $W^{1,2}$, and hence
    also in $C([0,T];
    L^2(X,\mu))$, which is to say $\norm{v(t) - v_n(t)}_{L^2} \to 0$
    uniformly in $t$.  Since $\Phi$ is Lipschitz, we also have
    $\norm{\Phi(v(t,\cdot)) - \Phi(v_n(t,\cdot))}_{L^2} \to 0$
    uniformly in $t$, and it follows that
    \begin{equation*}
      \int_0^T \norm{\Phi(v_n(t,\cdot))}_{L^2}^2 \chi'(t)\,dt \to 
      \int_0^T \norm{\Phi(v(t,\cdot))}_{L^2}^2 \chi'(t)\,dt.
    \end{equation*}

    We also have $v_n' \to v'$ in $L^2([0,T];\mathbb{D}^*)$ and $v_n
    \to v$ in $L^2([0,T]; \mathbb{D})$.  In particular, passing to a
    subsequence, we have $v_n(t) \to v(t)$ in $\mathcal{E}_1^{1/2}$-norm and
    $v_n'(t) \to v'(t)$ in $\mathbb{D}^*$-norm for
    almost every $t$.  Now by \cite[Theorem 1.4.2
    (v)]{fukushima-et-al}, we have $\varphi(v_n(t,\cdot)) \to
    \varphi(v(t,\cdot))$ $\mathcal{E}_1$-weakly for almost every $t$.
    Thus $(v_n'(t), \varphi(v_n(t, \cdot)))_{\mathbb{D}^*, \mathbb{D}}
    \to (v'(t), \varphi(v(t,\cdot)))_{\mathbb{D}^*, \mathbb{D}}$ for
    almost every $t$.  Now since $\varphi(v_n(t,\cdot))$ is a normal
    contraction of $v_n(t, \cdot)$, we have
    \begin{align*}
      \abs{(v_n'(t), \varphi(v_n(t, \cdot)))_{\mathbb{D}^*,
      \mathbb{D}}}
      \le \norm{v_n'(t)}_{\mathbb{D}^*}
      \sqrt{\mathcal{E}_1(\varphi(v_n(t, \cdot))) }
      \le
      \norm{v_n'(t)}_{\mathbb{D}^*} 
	   \sqrt{\mathcal{E}_1(v_n(t))}.
    \end{align*}
    Since $v_n \to v$ in $W^{1,2}$, it follows that the expression on
    the right side converges in $L^1([0,T])$.  Thus by a variant of
    the dominated convergence theorem, we have
    \begin{equation*}
       \int_0^T \chi(t) (v_n'(t), \varphi(v_n(t, \cdot)))_{\mathbb{D}^*,
     \mathbb{D}}\,dt \to \int_0^T \chi(t) (v'(t), \varphi(v(t, \cdot)))_{\mathbb{D}^*,
     \mathbb{D}}\,dt.
    \end{equation*}
    By passing to the limit, we have shown that (\ref{Phi-eq})
    holds for all $v \in W^{1,2}([0,T]; \mathbb{D}, \mathbb{D}^*)$,
    which implies the desired result.
\end{proof}




\begin{proof}[Proof of Theorem \ref{max-principle}]
  Let $\chi \in C^\infty_c((0,T))$, and set $\phi(t,x) = \chi(t)
  \varphi(u(t,x))$.  Since $\varphi$ is a normal contraction and $\varphi(u) = \varphi(u^+)$, we have
  $\phi(t) \in \mathbb{D}(U)$ for a.e. $t$.  Also, since
  $\mathcal{E}_1(\varphi(u(t))) \le \mathcal{E}_1(u(t))$, we have
  $\phi \in L^2([0,T]; \mathbb{D}(U))$, and $\phi(t)$ vanishes outside
  the support of $\chi$.  Thus (\ref{u-max-weak}) holds for $\phi$;
  that is,
  \begin{equation*}
      \int_0^T \chi(t) (u'(t), \varphi(u(t, \cdot)) )_{\mathbb{D}^*,
      \mathbb{D}}\,dt = -\int_0^T \chi(t) \mathcal{E}(u(t), \varphi(u(t, \cdot)))\,dt.
  \end{equation*}
  Now $\chi$ was arbitrary, so we must have $(u'(t), \varphi(u(t, \cdot)) )_{\mathbb{D}^*,
      \mathbb{D}} = -\mathcal{E}(u(t), \varphi(u(t)))$ for a.e. $t$.
  By Lemma \ref{Phi-lemma} and (\ref{gh43}), if we write $a(t) =
      \norm{\Phi(u(t, \cdot))}_{L^2}^2$, this says $a'(t) \le 0$.
      But $a(0) = 0$ since $u(0) \le 0$, and $a \ge 0$ by definition, so
      we must have $a = 0$ identically.  So we have $\Phi(u(t, \cdot))
      = 0$ a.e.  Since $\Phi(s) > 0$ for all $s > 0$, it must be that $u
      \le 0$ a.e.
\end{proof}

\bibliographystyle{plainnat}
\bibliography{allpapers}
\end{document}